\documentclass{article}

\usepackage{amsmath,amssymb,amsthm}
\usepackage[hidelinks]{hyperref}
\usepackage[utf8]{inputenc}
\usepackage{mathrsfs}
\usepackage{a4wide,graphicx,todonotes}
\usepackage{algorithm}
\usepackage{algpseudocode}
\usepackage{xcolor}

\theoremstyle{plain}
\newtheorem{theorem}{Theorem}[section]
\newtheorem{proposition}[theorem]{Proposition}
\newtheorem{lemma}[theorem]{Lemma}
\newtheorem{problem}{Problem}
\newtheorem{remark}[theorem]{Remark}
\usepackage{graphicx}

\newcommand{\muc}[2]{\mu_{#2,\S\ref{#1}}}
\newcommand{\pmuc}[2]{\mu_{#2,\S\ref{#1}}}

\numberwithin{equation}{section}

\newcommand{\scatam}{f_E}
\newcommand{\scatamu}[1]{f^{#1}_E}
\newcommand{\scataml}[1]{f_{E,#1}}
\newcommand{\scatamul}[2]{f^{#1}_{E,#2}}
\newcommand{\scatapprox}{\widetilde{f}_E^{\text{appr}}}

\newcommand{\dscatamu}[1]{\underline{f}^{#1}_E}

\newcommand{\dscatamul}[2]{\underline{f}^{#1}_{E,#2}}
\newcommand{\dscatapprox}{\widetilde{\underline{f}}_E^{\text{appr}}}
\renewcommand{\Re}{\mathop{\mathrm{Re}}}
\renewcommand{\Im}{\mathop{\mathrm{Im}}}
\DeclareMathOperator{\supp}{supp}
\DeclareMathOperator{\diag}{diag}
\DeclareMathOperator{\argmin}{argmin}
\newcommand{\spcind}{n}
\newcommand{\spcBd}{N}

\begin{document}
\begin{center}\LARGE An iterative approach to monochromatic phaseless inverse scattering\end{center}
\begin{center}\it A. D. Agaltsov\footnote{Max-Planck Institute for Solar Systems Research, 
Justus-von-Liebig-Weg 3, 37077 G\"ottingen, Germany 
  (\url{agaltsov@mps.mpg.de}).}, 
T. Hohage\footnote{Institute for Numerical and 
Applied Mathematics, University of G\"ottingen, Lotzestr. 16-18, 37083 G\"ottingen, Germany 
  (\url{hohage@math.uni-goettingen.de}) and Max-Planck Institute for Solar Systems Research.}, R. G. Novikov\footnote{CMAP, Ecole Polytechnique, CNRS, Universit\'e Paris-Saclay, 91128 Palaiseau, France; and IEPT RAS, 117997 Moscow, Russia
  (\url{novikov@cmap.polytechnique.fr}).}\end{center}
\begin{center} \today \end{center}

\begin{abstract}
This paper is concerned with the inverse problem to recover a compactly supported Schr\"odinger potential  
given the differential scattering cross section, i.e.\ the modulus, but not the 
phase of the scattering amplitude. To compensate for the missing phase information we assume 
additional measurements of the differential cross section in the presence of known background objects.  
We propose an iterative scheme for the numerical solution of this problem and prove that it converges 
globally of arbitrarily high order depending on the smoothness of the unknown potential as the energy 
tends to infinity. At fixed energy, however, the proposed iteration does not converge to the true solution even 
for exact data. Nevertheless, numerical experiments show that it yields remarkably accurate 
approximations with small computational effort even for moderate energies. 
At small noise levels it may be worth 
to improve these approximations by a few steps of a locally convergent iterative regularization method, 
and we demonstrate to which extent this reduces the reconstruction error. \medskip

\medskip

\end{abstract}

\section{Introduction}\label{sec.in}
In quantum mechanics  the interaction of an elementary particle 
at fixed energy $E>0$ with a macroscopic object  contained in a bounded domain $D$ is described by the 
Schr\"odinger equation 
\begin{subequations}
\begin{equation}\label{in.schreq}
    -\Delta \psi + v(x)\psi = E\psi, \quad x \in \mathbb R^d. 
\end{equation}
Here $\Delta$ is the standard Laplacian in $x$, and the potential function $v$ is assumed to satisfy 
\begin{equation}\label{in.vprop}
    v \in L^\infty(\mathbb R^d), \quad \supp v \subset D\subset \mathbb R^d, d\geq 2.
\end{equation}
Equation \eqref{in.schreq} can be also considered as the Helmholtz equation of acoustic and electrodynamic wave propagation at fixed frequency. 

For equation \eqref{in.schreq} we consider the classical scattering solutions $\psi^+(\cdot,k)$ of the 
form $\psi^+(x,k) = e^{ikx}+ \psi^{\rm s}(x,k)$ with a plane incident field $e^{ikx}$ such that $k \in \mathbb R^d$, $|k|^2=E$, 
and a scattered field $\psi^{\rm s}(\cdot,k)$ satisfying the Sommerfeld radiation condition 
\begin{equation}
|x|^{\frac{d-1}{2}}\left(\frac{\partial}{\partial |x|}-i|k|\right)\psi^{\rm s}(x,k) \to 0 \qquad \mbox{as }|x|\to \infty
\end{equation}
\end{subequations}
uniformly in $x/|x|$. This implies that $\psi^{\rm s}$ has the asymptotic behavior 
\begin{equation}\label{in.psi+as}
    \begin{aligned}
    &\psi^{\rm s}(x,k) = c(d,|k|) \frac{e^{i|k||x|}}{|x|^{\frac{d-1}{2}}} 
		\scatam\left(k,|k|\tfrac{x}{|x|}\right) 
		+ \mathcal{O}\left(|x|^{-\frac{d+1}{2}}\right),
	\quad |x|\to \infty,\\
	&c(d,|k|) := -\pi i (-2\pi i)^{(d-1)/2} |k|^{(d-3)/2}
	\end{aligned}
\end{equation}
with a function $\scatam$ called the scattering amplitude or
far field pattern at energy $E$. 
There are different conventions for the choice of the constant $c(d,|k|)$. The one above leads to the following 
simple asymptotic relation between the scattering amplitude and the inverse Fourier transform of $v$, 
see, e.g., \cite{Fad1956,Nov2015en}: 
\begin{gather}
    \widehat v(k-l) = \scatam(k,l) + \mathcal{O}(E^{-\frac 1 2}), \quad E \to + \infty, \label{in.vborn}\\ 
    \widehat v(p) := (2\pi)^{-d} \int_{\mathbb R^d} e^{ipx} v(x) \, dx, \quad p \in \mathbb R^d\label{in.fourier}
\end{gather}
$|\scatam(k,l)|^2$ is known as the differential scattering cross section for equation \eqref{in.schreq}. In quantum mechanics this quantity describes the probability density of scattering of particle with initial impulse $k$ into direction $l/|l| \neq k/|k|$, see, for example, \cite[Chapter 1, Section 6]{Fad1993}. Typically, the differential 
cross section is the only measurable quantity whereas the phase of the scattering amplitude cannot be 
determined directly by physical experiments. 
The problem of finding $v$ from $|\scatam|^2$ is known as the phaseless inverse scattering problem for equation \eqref{in.schreq}. Whereas the inverse scattering problem with phase information for equation \eqref{in.schreq}, i.e.\ 
the problem of finding $v$ from $\scatam$, has been studied intensively for a long time 
(see \cite{Alex2008,Barc2016,Bur2009en,Chad1989,Eskin2011,Fad1956,Grin2000,Haeh2001,Isay2013b,Isay2013c,Nov1988en,Nov1998,Nov2005b,Nov2013b,Nov2015en} and references therein),   
much less studies were performed in the phaseless case (see \cite{Agal2016b,Chad1989,Klib2014,Klib2016a,romanov2018phaseless,Nov2015LS,Nov2015b,Nov2016}).  

It is well known that the phaseless scattering data $|\scatam|^2$ does not determine $v$ uniquely even if $|\scatam|^2$ is given completely for all $E>0$; see, e.g., \cite{Nov2016}.
In the present work we continue studies of \cite{Nov2016,Agal2016b} assuming additional measurements 
of the following form: 
For the unknown $v$ satisfying \eqref{in.vprop} we consider additional a priori known background scatterers 
$w_1$, \dots, $w_L$ such that
\begin{equation}\label{in.wjass}
\begin{gathered}
        w_l \in L^\infty(\mathbb R^d), \quad \supp w_l \subset \Omega_l,\\
        \text{$\Omega_l$ is an open bounded domain in $\mathbb R^d$},\\
        w_l \neq 0, \quad \text{$w_{l} \neq w_{\tilde{l}}$ if $l \neq \tilde{l}$ (in $L^\infty(\mathbb R^d$)),}   
\end{gathered}
\end{equation}
where $l,\tilde{l} \in \{1,\ldots,L\}$.
In practice, we also typically have 
\[
\Omega_l \cap D = \varnothing,
\]
but this property will not be needed in our analysis.
We set
\begin{equation}
        S = \{ |\scatam|^2, |\scataml{1}|^2, \ldots, |\scataml{L}|^2\},
\end{equation}
where $\scatam$ is the scattering amplitude for $v$ at energy $E$, and $\scataml{1}$, \ldots, $\scataml{L}$ are the scattering amplitudes for $v_1$, \dots, $v_L$, where
\begin{equation}\label{in.defvj}
        v_l = v + w_l, \quad l = 1,\ldots, L.
\end{equation}
One can see that $S$ consists of the phaseless scattering data 
$|\scatam|^2$, $|\scataml{1}|^2$, \dots, $|\scataml{L}|^2$ measured sequentially, first, for the unknown scatterer $v$ and then for the unknown scatterer $v$ in the presence of known scatterer $w_l$ disjoint from $v$ for $l = 1$, \dots, $L$.
We consider the following inverse scattering problem without phase information for equation \eqref{in.schreq}:

\begin{problem}\label{in.probpisc} Reconstruct coefficient $v$ from the phaseless scattering data $S$ for some appropriate background scatterers $w_1$, \ldots, $w_L$.
\end{problem}

In this paper we propose an iterative approach to Problem \ref{in.probpisc} with iterates $u^{(j)}_E$, 
$j=1,2,\dots$ and prove error 
bounds of the form 
\begin{equation}\label{eq:error_bound}
\|u^{(j)}_E-v\|_{L^\infty} = \mathcal{O}\left(E^{-\alpha_j}\right)
\end{equation}
with exponents $\alpha_j$ tending to $\infty$ as $j\to \infty$ for infinitely smooth potentials $v$. 

For the inverse scattering problem with phase information such a substantial improvement of 
the Born approximation, which serve as first iterate $u_1$ (see \eqref{in.vborn}) has been obtained 
in \cite{Nov2015en}, and first numerical tests were reported in \cite{Barc2016}. 

Studies on Problem \ref{in.probpisc} in dimension $d=1$ for $L = 1$ were started in \cite{Akto1998}, where phaseless scattering data was considered for all $E > 0$. Note also that a phaseless optical imaging in the presence of known background objects was considered, in particular, in \cite{Gov2009}. Studies on Problem \ref{in.probpisc} in dimension $d \geq 2$ were started in \cite{Nov2016} and continued recently in \cite{Agal2016b}.
The key result of \cite{Nov2016} consists in a proper extension of formula \eqref{in.vborn} for the Fourier transform $\widehat v$ of $v$ to the phaseless case of Problem \ref{in.probpisc}, $d \geq 2$, which will be 
detailed in Section \ref{sec.pex}.
The main results of \cite{Agal2016b} consist in proper extensions of formula \eqref{in.utov} in the configuration space to the case of Problem \ref{in.probpisc} for $d \geq 2$; see also Section \ref{sec.pex}.
However, the convergence of the approximations to $v$ as $E \to +\infty$ in \cite{Agal2016b} is slow, 
in particular, the exponent $\alpha$ in \eqref{eq:error_bound} is always $\leq \frac{1}{2}$. 


In addition, our theoretical iterative monochromatic reconstructions for Problem \ref{in.probpisc} are illustrated numerically in Section \ref{sec.num}.

\section{Iterative inversion with phase information}\label{sec.prelim}

\subsection{Inverse scattering with phase information}\label{sse.pha} 
Recall that the scattering amplitude $\scatam$ is defined on the set 
\begin{equation}
         \mathcal M_E = \bigl\{ (k,l) \in \mathbb R^d \times \mathbb R^d \colon |k|^2 = |l|^2 = E \bigr\}.
\end{equation}
In view of \eqref{in.vborn} we assume that the scattering amplitude, 
and later the differential 
cross section is defined on some subset $\mathcal{M}^{\mathrm m}_E\subset \mathcal M_E$ 
such that the function 
\begin{equation}\label{eq:defiPhi}
\tilde{\Phi}\colon\mathcal{M}^{\mathrm m}_E\to \mathcal B^d_{2\sqrt E}\qquad  
\tilde{\Phi}(k,l):=k-l
\end{equation}
is surjective. Here and in the following $\mathcal B^d_r$ denotes the closed  ball 
\begin{equation}
   \mathcal B^d_r = \bigl\{ p \in \mathbb R^d \colon |p| \leq r \bigr\}, \quad r > 0. \label{pex.defBr}
\end{equation}
For $d \geq 2$ we may construct a $d$-dimensional subset $\mathcal{M}^{\mathrm m}_E
\subset \mathcal M_E$ such that $\tilde{\Phi}$ is bijective as follows: 
Let us choose a piece-wise continuous function 
$\gamma:\mathbb \mathcal B^d_{2\sqrt{E}}\to \mathbb R^d$ such that 
$|\gamma(p)| = 1$ and $\gamma(p) p = 0$ for all $p\in  \mathcal B^d_{2\sqrt{E}}$ and set 
\begin{gather}
\mathcal{M}^{\mathrm m}_E = \left\{\left(
 \tfrac p 2 + \bigl(E - \tfrac{p^2}{4}\bigr)^{1/2} \gamma(p), 
-\tfrac p 2 + \bigl(E - \tfrac{p^2}{4}\bigr)^{1/2} \gamma(p)
\right) \colon p \in \mathcal B^d_{2\sqrt E} \right\}. \label{in.defGE}
    \end{gather}
To use the Born approximation \eqref{in.vborn} and its 
refinements if $\tilde{\Phi}$ is not injective, we average over the set 
$\tilde\Phi^{-1}(p)$. To this end we assume that for all $p\in\mathcal{B}^d_{2\sqrt{E}}$ the set $\tilde{\Phi}^{-1}(p)$ is a piecewise smooth manifold of size 
$\left|\tilde\Phi^{-1}(p)\right|$ and define the averaging operator 
\begin{equation}\label{eq:defPhi}
\begin{aligned}
&\Phi : L^1(\mathcal{M}^{\mathrm m}_E) \to L^1\left(\mathcal B^d_{2\sqrt{E}}\right),\qquad 
(\Phi f)(p):= \tfrac{1}{\left|\tilde{\Phi}^{-1}(p)\right|}
\int_{\tilde{\Phi}^{-1}(p)} \scatam(k,l)\,d(k,l).
\end{aligned}
\end{equation}
Using this mapping we can define an approximation $u_E$ to $v$ on $D$ by 
\begin{equation}\label{eq:defi_uE}
u_E(x):=\int_{\mathcal B^d_{2\sqrt{E}}} e^{-ipx}(\Phi f)(p)\,dp, \qquad x\in D. 
\end{equation}
Let $W^{n,1}(\mathbb R^d)$ denote the Sobolev space of $n$-times smooth functions in the sense of $L^1(\mathbb R^d)$:
\begin{equation}\label{in.Wn1def}
  \begin{aligned}
   & W^{n,1}(\mathbb R^d) := \bigl\{ u \in L^1(\mathbb R^d) \colon \|u\|_{n,1} < \infty \bigr\}\quad\mbox{with} \\
   & \|u\|_{n,1} := \max\limits_{|J| \leq n} \left\| \frac{\partial^{|J|} u}{\partial x^J}\right\|_{L^1(\mathbb R^d)}, \quad n \in \mathbb N \cup \{0\}.
  \end{aligned}
\end{equation}
If $v \in W^{n,1}(\mathbb R^d)$, $n > d$, in addition to the initial assumptions \eqref{in.vprop}, then $u_E$ satisfies the error bound 
\begin{equation}\label{in.utov}
    \|u_E-v\|_{L^{\infty}(D)} = \mathcal{O}(E^{-\alpha}) \quad  
		\mbox{as $E\to+\infty$}\qquad \mbox{with } \alpha = \frac{n-d}{2n}
\end{equation}
for $d \geq 2$; see, for example, \cite{Nov2015en}. 
Essential improvements of the approximation $u_E$ in \eqref{eq:defi_uE} 
were achieved in \cite{Nov1998,Nov2005b,Nov2015en}. In particular, formula \eqref{in.utov} was principally improved in \cite{Nov2015en} by constructing iteratively nonlinear approximate reconstructions $u^{(j)}_E$ such that 
$u^{(1)}_E=u_E$ and 
\begin{equation}\label{in.ujtov}
    \begin{gathered}
        \|u^{(j)}_E - v\|_{L^{\infty}(D)} = \mathcal{O}(E^{-\alpha_j}) \quad 
				\text{with }
        \alpha_j = \tfrac{n-d}{2d} \left(1 - \left(\tfrac{n-d}{n}\right)^j\right) , \quad j \geq 1,
    \end{gathered}
\end{equation}
as $E\to+\infty$ for $d \geq 2$ 
if $v \in W^{n,1}(\mathbb R^d)$, $n > d$, in addition to the initial assumptions \eqref{in.vprop}. The point is that 
\begin{equation}\label{in.alphalim}
    \begin{aligned}
    & \alpha_j \to \alpha_\infty = \tfrac{n-d}{2d} & \text{as $j \to +\infty$}, \\
    & \alpha_j \to \tfrac j 2 & \text{as $n \to +\infty$},\\
    & \alpha_\infty \to +\infty & \text{as $n \to +\infty$};
    \end{aligned}
\end{equation}
that is the convergence in \eqref{in.ujtov} as $E \to +\infty$ is drastically better than the convergence in \eqref{in.utov}, at least, for large $n$ and $j$. 

\subsection{Iterative step for phased inverse scattering}\label{sec.itp}
Recall that the outgoing fundamental solution to the Helmholtz equation is given by 
\[ 
G^+(x,k) = -(2\pi)^{-d} \int_{\mathbb R^d} \frac{e^{i\xi x}d\xi}{\xi^2-k^2-i0}
= \frac i 4 \big( \frac{k}{2\pi |x-y|} \big)^\nu H^{(1)}_\nu (k|x-y|)
\quad\mbox{with}\quad \nu := \tfrac d 2 - 1
,\label{in.G+def}
\]
where $H^{(1)}_{\nu}$ denotes the Hankel function of the first kind of order $\nu$. 
Let $\mathcal{G}^+(k)$ denote the convolution operator with kernel $G^+(\cdot,k)$. 
The following estimate, which goes back to \cite{Agmon1975}, is essential for studies on direct scattering (see, e.g., \cite{Eskin2011} (\S 29), \cite{Nov2015en}, and references therein) and will also be crucial for our 
analysis:  
\begin{equation}\label{in.Agmon}
  \begin{gathered}
  \| \Lambda^{-s} \mathcal{G}^+(k) \Lambda^{-s} \|_{L^2(\mathbb R^d) \to L^2(\mathbb R^d)} \leq a_0(d,s) |k|^{-1}, \\
  k \in \mathbb R^d, \; |k| \geq 1, \quad \text{for $s > \tfrac 1 2$}.
  \end{gathered}
\end{equation}
Here $\Lambda^{-s}$ denotes the operator of multiplication by $(1+|x|^2)^{-s/2}$. 

Let $v$, satisfying \eqref{in.vprop}, be the unknown potential and 
$v^*_E$ be an approximation to $v$, and assume that there exist constants 
$A,E^*,K,\alpha>0$ such that 
\begin{subequations}\label{eqs:itp.v*}
\begin{align}
 & v^*_E \in L^\infty(\mathbb R^d), \quad \supp v^*_E \subset D \label{itp.v*prop}\\
  &\|v^*_E-v\|_{L^{\infty}(D)} \leq A E^{-\alpha}\label{itp.v*est}\\
	&\|v\|_{L^\infty(\mathbb R^d)} \leq K, \quad \|v^*_E\|_{L^\infty(\mathbb R^d)} \leq K \label{itp.vv*leM}
\end{align}
\end{subequations}
for all $E \geq E^*$.

For inverse scattering with phase information the iterative step of \cite{Nov2015en} is based on the following lemma.

\begin{lemma}\label{itp.lemma} Let $v$, satisfying \eqref{in.vprop}, be the unknown potential and $v^*(\cdot,E)$ be an approximation to $v$ satisfying \eqref{eqs:itp.v*} for some $A > 0$, $\alpha \geq 0$, $K>0$ and $E^* = E^*(K,D)$, where
\begin{equation}\label{itp.E*def}
  \begin{gathered}
	E^*(K,D) = 2 K a_0(d,s) \sup_{x \in D} (1+|x|^2)^{s / 2},\\
	\text{for some $s > \tfrac 1 2$, where $a_0(d,s)$ is the constant of \eqref{in.Agmon}}.
  \end{gathered}
\end{equation}
Let $\scatam$, $\scatamu{*}$ be the scattering amplitudes of $v$, $v^*_E$. Then 
there exists a constant $\muc{sec.itp}{1} = \muc{sec.itp}{1}(A,K,D) > 0$ such that 
\begin{equation}\label{itp.ff*vv*est}
    \sup_{(k,l) \in \mathcal M_E}|\scatam(k,l) - \scatamu{*}(k,l) + \widehat{v^*_E}(k-l) - \widehat v(k-l)| \leq \muc{sec.itp}{1} E^{-\alpha - \frac 1 2},
    \quad E \geq E^*,
\end{equation}
 where $\widehat v$, $\widehat{v^*_E}$ are the Fourier transforms of $v$, $v^*_E$ defined according to \eqref{in.fourier}.
\end{lemma}

Note that in this paper we use the notation $\mu_{k,\S X}$, $k \geq 1$, for the constants of Section X. 

Lemma \ref{itp.lemma} follows from Lemma 3.2 of \cite{Nov2015en} for $v_0 \equiv 0$, where $v_0$ is the background potential of \cite{Nov2015en}. The proof of Lemma 3.2 of \cite{Nov2015en} essentially uses estimate \eqref{in.Agmon}.

In particular, due to \eqref{itp.ff*vv*est}, the function 
\begin{equation}\label{itp.U**def}
  U^{**}_E := \Phi\scatam - \Phi \scatamu{*} + \widehat{v^*_E},
\end{equation}
where $\Phi\scatam$, $\Phi\scatamu{*}$ are defined according to \eqref{eq:defPhi}, satisfies the following improved error estimate 
compared to \eqref{itp.v*est}: 
\begin{gather}\label{itp.U**est}
    \|U^{**}_E - \widehat v\|_{L^{\infty}(\mathcal B^d_{2\sqrt E})} 
		\leq \muc{sec.itp}{1} E^{-\alpha - \frac 1 2}, \qquad  E \geq E^*.
\end{gather}
If $v \in W^{n,1}(\mathbb R^d)$, $n > d$ (in addition to the initial assumptions \eqref{in.vprop}), and if
\begin{equation}\label{itp.alphaas}
  \alpha < \tfrac{n}{2d} - \tfrac 1 2,
\end{equation}
then this permits to construct an improved approximation $v^{**}_E$ to the unknown potential $v$ as follows:
\begin{gather}
    \begin{gathered}\label{itp.v**def}
    v^{**}_E(x) := \begin{cases}
		\int_{\mathcal B^d_{r(E)}} e^{-ipx} U^{**}_E(p) \, dp,& x \in D,\\
     0, &x \not \in D,
		\end{cases}
	\qquad	r(E) := 2\tau E^{\frac{1+2\alpha}{2n}}, \quad \tau \in (0,1],
    \end{gathered}
		\end{gather}
		Here $U^{**}_E$ is defined in \eqref{itp.U**def} and $E^* = E^*(K,D)$ is the constant of \eqref{itp.E*def}. It follows that there exists 
		a constant $B>0$ such that 
  \begin{gather}\label{itp.v**est}
  \|v^{**}_E-v\|_{L^{\infty}(D)} \leq B E^{-\beta}, \quad E \geq E^* 
	\mbox{ with }
  \beta := \alpha (1 - \tfrac d n) + \tfrac 1 2 - \tfrac{d}{2n}, 
\end{gather}
Note that 
\[
\alpha < \beta < \tfrac{n}{2d} - \tfrac 1 2
\]
and that condition \eqref{itp.alphaas} implies that $r(E) \leq 2\sqrt E$, so that the definition \eqref{itp.v**def} is correct. 

\section{Iterative inversion from phaseless data}
\subsection{Low order potential reconstruction formulas from phaseless data}\label{sec.pex}
In this subsection we extend the formulas \eqref{in.vborn} and 
\eqref{in.utov} to the phaseless case. 
The key result of \cite{Nov2016} consists in the following formulas for solving Problem \ref{in.probpisc} in dimension $d \geq 2$ for $L = 2$ at high energies $E$:  
\begin{equation}\label{pex.pBornAbs}
    \begin{gathered}
        \sup_{p \in \mathcal B^d_{2\sqrt{E}}}\left| 
				|\widehat{v_{l}}(p)|^2 - |\Phi\scataml{l}(p)|^2\right| 
				= \mathcal{O}(E^{-\frac 1 2}), \quad E \to + \infty,  \quad l = 0,1,2,
    \end{gathered}
\end{equation}
where $v_0 = v$, $v_l$ is defined by \eqref{in.defvj}, $l = 1$, $2$, 
$\scataml{0} = \scatam$, $\scataml{1}$, $\scataml{2}$ are the scattering amplitudes for $v_0$, $v_1$, $v_2$, respectively; in addition, 
\begin{equation}\label{pex.pBorn}
    \begin{pmatrix}
       \Re \widehat v \\
       \Im \widehat v
    \end{pmatrix} = 
    \frac 1 2 \begin{pmatrix} 
                \Re \widehat{w_1} & \Im \widehat{w_1} \\
                \Re \widehat{w_2} & \Im \widehat{w_2}
              \end{pmatrix}^{-1}
    \begin{pmatrix}
        |\widehat{v_1}|^2 - |\widehat v|^2 - |\widehat{w_1}|^2 \\
        |\widehat{v_2}|^2 - |\widehat v|^2 - |\widehat{w_2}|^2
    \end{pmatrix},
\end{equation}
where $\widehat v = \widehat v(p)$, $\widehat{v_l} = \widehat v_l(p)$, $\widehat{w_l} = \widehat w_l(p)$, $p \in \mathbb R^d$, and formula \eqref{pex.pBorn} is considered for all $p$ such that the determinant
\begin{equation}\label{pex.defzeta}
   \zeta_{\widehat{w_1},\widehat{w_2}}(p) := \Re \widehat{w_1}(p) \Im \widehat{w_2}(p) - \Im \widehat{w_1}(p) \Re \widehat{w_2}(p) \neq 0.
\end{equation}
Formulas \eqref{pex.pBornAbs}, \eqref{pex.pBorn} can be considered as a natural extension of formula \eqref{in.vborn} to the phaseless case of Problem \ref{in.probpisc}, $d \geq 2$, $L=2$ and lead to the function 
$\mathrm{Urec}$ defined in Algorithm \ref{alg:Urec} for the approximate 
reconstruction $U_{\widehat{w_1},\widehat{w_2}}$ of $\widehat{v}$:
\begin{equation}\label{eq:Urec}
U_{\widehat{w_1},\widehat{w_2}}(p) := 
\mathrm{Urec}\left(\widehat{w_1}(p),\widehat{w_2}(p),
\Phi|\scatam|^2(p),
\Phi|\scatamul{}{1}|^2(p), \Phi|\scatamul{}{2}|^2(p)\right),\qquad
|p|\leq 2\sqrt{E}
\end{equation}

\begin{algorithm}[ht]
\caption{$\mathrm{function}\; U = \texttt{Urec}(W_1, W_2, F,F_1,F_2)$}
\label{alg:Urec}
{\bf data:} 
Fourier transforms of reference potentials at 
some point $p$: $W_1=\widehat{w}_1(p)$, $W_2=\widehat{w}_2(p)$; \\
scattering amplitude at $(k,l)$, $k-l=p$ without reference potential: $F=|f(k,l)|^2$\\
scattering amplitudes with reference potentials: $F_1=|f_1(k,l)|^2$, $F_2 = |f_2(k,l)|^2$\\
{\bf result:} approximation to the Fourier transform of the unknown potential $v$ at $p$: 
$U\approx \widehat{v}(p)$
\begin{algorithmic}

\State $M:= \left(\begin{array}{cc}
\Re W_1 & \Im W_1 \\ \Re W_2 & \Im W_2 
\end{array}\right)$; $b:= \left(\begin{array}{c}F_1 - F -|W_1|^2\\
F_2-F-|W_2|^2\end{array}\right)$; $\left(\begin{array}{c} x\\y\end{array}\right):=\frac{1}{2}M^{-1}b$; $U:=x+iy$
\end{algorithmic}
\end{algorithm}


On the level of analysis (e.g., error estimates), the principal complication of \eqref{pex.pBornAbs}, \eqref{pex.pBorn} in comparison with \eqref{in.vborn} consists in possible zeros of the determinant $\zeta_{\widehat{w_1},\widehat{w_2}}$ of \eqref{pex.defzeta}. This complication is, in particular, essential if one tries to transform \eqref{pex.pBornAbs}, \eqref{pex.pBorn} into an approximate reconstruction in the configuration space, applying the inverse Fourier transform to $\widehat v = \Re \widehat v + i \Im \widehat v$ of \eqref{pex.pBorn}. For some simplest cases, the results of transforming \eqref{pex.pBornAbs}, \eqref{pex.pBorn} to approximate reconstructions in the configuration space, including error estimates, were given in \cite{Agal2016b}.

{\bf Background potentials of type A:} 
The first simplest case analyzed in \cite{Agal2016b} is 
\begin{align}\label{pex.defw1w2c}
    &w_1(x) := w(x-T_1), \quad  w_2(x) := i w(x-T_1), \quad x \in \mathbb R^d,\\
    &\begin{aligned}\label{pex.defw}
    &\mbox{for some } w \in C(\mathbb R^d) \mbox{ such that } 
		w = \overline w, \; \text{$w(x) = 0$ for $|x|>R$}, \mbox{and}\\
    &\widehat w(p) = \overline{\widehat w(p)} \geq \muc{sec.pex}{1} (1+|p|)^{-\sigma}, \quad p \in \mathbb R^d,
    \end{aligned} 
\end{align}
for some fixed $T_1 \in \mathbb R^d$, $R > 0$, $\muc{sec.pex}{1} > 0$, $\sigma > d$, where $T_1$ and $R$ are chosen in such a way that $w_1$ satisfies \eqref{in.wjass} (and, as a corollary, $w_1$, $w_2$ satisfy \eqref{in.wjass} with $\Omega_2 = \Omega_1$). In addition, a broad class of $w$ satisfying \eqref{pex.defw} was constructed in Lemma 1 of \cite{Agal2016b}. One can see that
\begin{equation}\label{pex.zetac}
    \begin{gathered}
        \zeta_{\widehat{w_1},\widehat{w_2}}(p) = |\widehat w(p)|^2 \geq \pmuc{sec.pex}{1}^2 (1+|p|)^{-2\sigma}, \quad p \in \mathbb R^d, \\
        \text{if $w_1$, $w_2$ are defined by \eqref{pex.defw1w2c}, \eqref{pex.defw}}.
    \end{gathered}
\end{equation}

{\bf Background potentials of type B:} 
The second simplest case analyzed in \cite{Agal2016b} is 
\begin{equation}\label{pex.defw1w2r}
    \begin{gathered}
    w_1(x) = w(x-T_1), \quad w_2(x) = w(x-T_2), \quad x \in \mathbb R^d, \\
    \text{for some fixed $T_1, T_2 \in \mathbb R^d$},
    \end{gathered}
\end{equation}
where $w$ is the same as in \eqref{pex.defw}, and $T_1$, $T_2$, $R$ are chosen in such a way that $w_1$, $w_2$ satisfy \eqref{in.wjass}. One can see that
\begin{equation}\label{pex.zetar}
    \begin{gathered}
        \zeta_{\widehat{w_1},\widehat{w_2}}(p) = \sin(py) |\widehat w(p)|^2, \quad y = T_2 - T_1 \neq 0, \quad p \in \mathbb R^d, \\
        |\zeta_{\widehat{w_1},\widehat{w_2}}(p)| \geq \pmuc{sec.pex}{1}^2 (1+|p|)^{-2\sigma} \tfrac{2\varepsilon}{\pi}, \quad p \in \mathbb R^d \setminus Z^\varepsilon_y,\\
        \text{if $w_1$, $w_2$ are defined by \eqref{pex.defw}, \eqref{pex.defw1w2r},}
    \end{gathered}
\end{equation}
where
\begin{equation}\label{pex.defZe}
   Z^\varepsilon_y = \bigl\{ p \in \mathbb R^d \colon py \in (-\varepsilon,\varepsilon) + \pi \mathbb Z \bigr\}, \quad y \in \mathbb R^d \setminus 0, \;\; 0 < \varepsilon < 1.
\end{equation}

First consider background potentials $w_1$, $w_2$ of type A 
(see \eqref{pex.defw1w2c}, \eqref{pex.defw}) and assume that 
$v$ satisfies \eqref{in.vprop}, $v \in W^{n,1}(\mathbb R^d)$ for some $n>d$.  Then the result of transforming $U_{\widehat{w_1},\widehat{w_2}}$ in 
\eqref{eq:Urec} by 
\begin{equation}\label{pex.defuc}
    \begin{gathered}
    u_E(x) := \int_{\mathcal B^d_{r(E)}} e^{-ipx} U_{\widehat{w_1},\widehat{w_2}}(p,E) \, dp, \quad x \in \mathbb R^d, \\
    r(E) = 2\tau E^\frac{\alpha}{n-d} \quad \text{for some fixed $\tau \in (0,1]$},
    \end{gathered}
\end{equation}
to an approximate reconstruction in the configuration space is as follows 
(see \cite[Theorem 1]{Agal2016b}):
\begin{equation}\label{pex.ucerr}
    \begin{gathered}
    \|u_E- v\|_{L^\infty(\mathbb R^d)} = \mathcal{O}(E^{-\alpha}) \quad \; E \to + \infty\quad \mbox{with }
    \alpha = \frac{\frac{1}{2}(n-d)}{n+\sigma}.
    \end{gathered}
\end{equation}
Now consider background potentials $w_1$, $w_2$ of type B 
(see \eqref{pex.defw1w2r}) and assume again that $v$ satisfies \eqref{in.vprop}, $v \in W^{n,1}(\mathbb R^d)$ for some $n>d$. 
We transform $U_{\widehat{w_1},\widehat{w_2}}$ in \eqref{eq:Urec} 
to an approximation $u_E$ in the configuration space as follows:
\begin{equation}\label{pex.defush}
  \begin{gathered}
   u_E(x) = u_{E,1}(x) + u_{E,2}(x), \quad x \in \mathbb R^d, \\
  u_{E,1}(x) = \int\limits_{ \mathcal B^d_{r(E)} \setminus Z^{\varepsilon(E)}_y} e^{-ipx} U_{\widehat{w_1},\widehat{w_2}}(p,E) \, dp,  \\
  u_{E,2}(x) = \int\limits_{ \mathcal B^d_{r(E)} \cap Z^{\varepsilon(E)}_y} e^{-ipx}  U^{\varepsilon(E)}_{\widehat{w_1},\widehat{w_2}}(p,E) \, dp, \\
  r(E) = 2 \tau E^{\frac{\alpha}{n-d}}, \quad \varepsilon(E)= E^{-\tfrac{\alpha}{2}} \quad \text{for some $\tau \in (0,1]$,}
  \end{gathered}
\end{equation}
$Z^\varepsilon_y$ is defined in \eqref{pex.defZe}, and 
\begin{gather}\label{pex.Uedef}
 U^\varepsilon_{\widehat{w_1},\widehat{w_2}}(p,E) = \tfrac 1 2 \bigl( U_{\widehat{w_1},\widehat{w_2}}(p^\varepsilon_-,E)+U_{\widehat{w_1},\widehat{w_2}}(p^\varepsilon_+,E) \bigr),\\
 \begin{gathered}\label{pex.pbotdef}
  p^\varepsilon_\pm = p_\bot + \pi z(p) \tfrac{y}{|y|^2} \pm \varepsilon \tfrac{y}{|y|^2}, \quad p_\bot = p - (p \cdot y) \tfrac{y}{|y|^2}, \quad p \in \mathcal B^d_{2\sqrt E} \cap Z^\varepsilon_y, \\
  \text{for the unique $z(p) \in \mathbb Z$ such that $|py - \pi z(p)| < \varepsilon$}.
 \end{gathered}
\end{gather}
Then it was shown in  \cite[Theorem 2]{Agal2016b} that
\begin{equation}\label{pex.usherr}
    \begin{gathered}
    \|u_E- v\|_{L^\infty(\mathbb R^d)} = \mathcal{O}(E^{-\alpha}) \quad \; 
		E \to + \infty\quad \mbox{with }
    \alpha = \frac{\frac{1}{2}(n-d)}{n+\sigma+\frac{n-d}{2}},
    \end{gathered}
\end{equation}
The geometry of vectors $p$, $p_\bot$, $y$, $p^\varepsilon_\pm$ is illustrated in Fig.~\ref{fig.defppm} for the case when the direction of $y$ coincides with the basis vector $e_1 = (1,0,\dots,0)$.
\begin{figure}[h]
\begin{center}
\includegraphics[width=.5\linewidth]{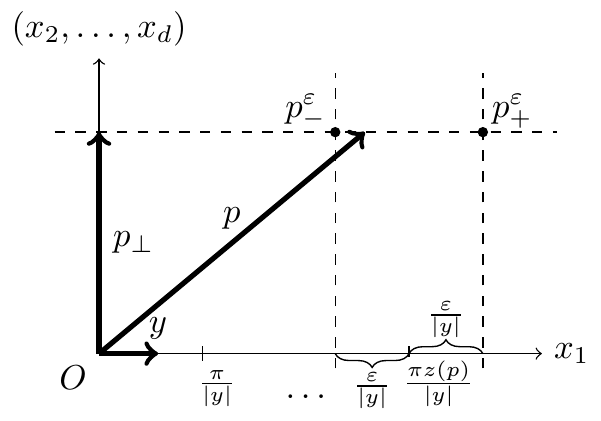}
\end{center}
\caption{Vectors $p$, $p_\bot$, $y$ and $p_\pm^\varepsilon$ of formula \eqref{pex.pbotdef}}
\label{fig.defppm}
\end{figure}

\subsection{Approximate reconstruction of phased scattering data}\label{sec.isp}

We consider Problem \ref{in.probpisc} for $d \geq 2$, $L = 2$, with the unknown potential $v$ satisfying \eqref{in.vprop} and with the background potentials $w_1$, $w_2$ satisfying \eqref{in.wjass}. Let 
\begin{equation}\label{isp.Dextdef}
  D_\text{ext} = D \cup \Omega_1 \cup \Omega_2,
\end{equation}
where $D$, $\Omega_1$, $\Omega_2$ are the domains in \eqref{in.vprop} and \eqref{in.wjass}.

Let $v^*_E$ be an approximation to $v$ satisfying \eqref{eqs:itp.v*} for some $A > 0$, $\alpha \geq 0$, $K > 0$ and for $E^* = E^*(K,D_\text{ext})$, where $E^*(K,D_\text{ext})$ is defined according to \eqref{itp.E*def}. 
In addition,  we suppose that
\begin{equation}\label{isp.v*prop}
  \|v+w_l\|_{L^\infty(\mathbb R^d)} \leq K, \; \|v^*(\cdot,E)+w_l\|_{L^\infty(\mathbb R^d)} \leq K, \; E \geq E^*, \; l = 1,2.
\end{equation}
Using the scattering amplitudes $\scatamu{*}$, $\scatamul{*}{1}$, 
and $\scatamul{*}{2}$ of the known potentials $v^*_E$, 
$v^*_{E,1} := v^*_E+w_1$, and $v^*_{E,2} := v^*_E+w_2$, respectively, 
and the phaseless scattering data $S$ of Problem \ref{in.probpisc}, we construct an approximation $\scatapprox(p)$ to $(\Phi\scatam)(p)$ 
for $|p|\leq 2\sqrt{E}$ by the function in Algorithm \ref{alg:Urec} as follows:
\begin{equation}\label{isp.fappr}
\scatapprox(p) := \mathrm{Urec}\left(\Phi\scatamul{*}{1}(p)-\Phi\scatamu{*}(p),\Phi\scatamul{*}{2}(p)-\Phi\scatamu{*}(p),
\Phi|\scatam|^2(p),
\Phi|\scatamul{}{1}|^2(p), 
\Phi|\scatamul{}{2}|^2(p)\right)
\end{equation}
Note that $\scatapprox(p)$ is well defined if 
\begin{equation}\label{eq:zeta*}
\zeta^*(p,E):= \det \begin{pmatrix}
\Re(\Phi\scatamul{*}{1}(p)-\Phi\scatamu{*}(p)) & 
\Im(\Phi\scatamul{*}{1}(p)-\Phi\scatamu{*}(p)) \\
\Re(\Phi\scatamul{*}{2}(p)-\Phi\scatamu{*}(p)) & 
\Im(\Phi\scatamul{*}{2}(p)-\Phi\scatamu{*}(p))
\end{pmatrix} \neq 0.
\end{equation}
Note that condition \eqref{eq:zeta*} is satisfied for sufficiently large $E$ at fixed $p \in \mathcal B^d_{2\sqrt E}$ if $\zeta_{\widehat w_1,\widehat w_2}(p) \neq 0$, where $\zeta_{\widehat w_1,\widehat w_2}$ is the determinant of formula \eqref{pex.defzeta}.
This follows from the estimate
\begin{equation}\label{isp.zetaest}
  \begin{gathered}
  | \zeta^*(p,E) - \zeta_{\widehat{w_1},\widehat{w_2}}(p) | \leq \muc{sec.isp}{1} E^{-\frac 1 2}, \quad p \in \mathcal B^d_{2\sqrt E}, \; E \geq E^*,
  \end{gathered}
\end{equation}
where $\muc{sec.isp}{1} > 0$ and $E^* = E^*(K,D_\text{ext})$ is defined according to \eqref{itp.E*def}, \eqref{isp.Dextdef}. Estimate \eqref{isp.zetaest} follows from the definition of $M$ in Algorithm \ref{alg:Urec}, 
the formula $M^{-1} = \frac{1}{\zeta}
\left(\begin{smallmatrix} \Im W_2 & - \Im W_1\\ -\Re W_2&\Re W_1\end{smallmatrix}\right)$,  and from the estimates
\begin{equation}\label{isp.Bornest}
  \begin{aligned}
  \bigl| \widehat v(p) - \Phi\scatam(p) \bigr| & \leq \muc{sec.isp}{2} E^{-\frac 1 2}, 
	\quad p \in \mathcal B^d_{2\sqrt{E}},\; E \geq E^*,\\
  \bigl| \widehat v_l(p) - \Phi\scataml{l}(p) \bigr| & \leq \muc{sec.isp}{2} E^{-\frac 1 2}, 
	\quad p \in \mathcal B^d_{2\sqrt{E}},\; E \geq E^*, \; l = 1,2,
  \end{aligned}
\end{equation}
where $\muc{sec.isp}{2} = \muc{sec.isp}{2}(K,D_\text{ext}) > 0$, $E^* = E^*(K,D_\text{ext})$; see, e.g., \cite{Nov2015en}.

\begin{lemma}\label{isp.lemmadisp} Let $v$, satisfying \eqref{in.vprop}, be the unknown potential of Problem \ref{in.probpisc} for $d \geq 2$, $L = 2$. Let $w_1$, $w_2$ be the same as in \eqref{in.wjass}. Let $v^*(\cdot,E)$ be an approximation to $v$, satisfying \eqref{eqs:itp.v*} for some $A > 0$, $\alpha \geq 0$, $K>0$ and for $E^* = E^*(K,D_\text{ext})$ defined according to \eqref{itp.E*def}, \eqref{isp.Dextdef}. Let $p \in \mathcal B^d_{2\sqrt E}$ be such that
\begin{equation}\label{isp.zetageqeps}
  |\zeta_{\widehat{w_1},\widehat{w_2}}(p)| \geq \delta,
\end{equation}
for some fixed $\delta > 0$. Then:
\begin{gather}
  |\zeta^*(p,E)| \geq \tfrac \delta 2, \quad E \geq E^{**}, \label{isp.zetaEgeqeps}\\
  \left|\Phi\scatam(p) - \scatapprox(p) \right| \leq \muc{sec.isp}{3} \delta^{-1} E^{-\alpha-\frac 1 2}, \quad E \geq E^{**},   \label{isp.dispest} \\
  E^{**} = \max \bigl( 4 \pmuc{sec.isp}{1}^2 \delta^{-2}, E^* \bigr),\label{isp.E**def}
\end{gather}
where $\scatapprox$ is defined by \eqref{isp.fappr}, $\zeta^*$ is defined by \eqref{eq:zeta*},
$\muc{sec.isp}{3}$ is defined in \eqref{dis.mu4def} and $\muc{sec.isp}{1}$ is the constant of \eqref{isp.zetaest}.
\end{lemma}
Lemma \ref{isp.lemmadisp} is proved in Section \ref{sec.dis}.

The point is that the right-hand side of the estimate in \eqref{isp.dispest} decays faster than the right-hand side of the estimate in \eqref{itp.v*est} as $E \to + \infty$. This is a crucial advantage of $\scatapprox$ as an approximation to the unknown phased scattering data $\Phi\scatam$ in comparison with $\Phi\scatamu{*}$.

Using Lemma \ref{isp.lemmadisp} we construct the iterative step for phaseless inverse scattering, see Sections \ref{sec.itc} and \ref{sec.itr}.

\subsection{Iterations for background potentials of type A}\label{sec.itc}
In this subsection we consider background potentials $w_1$, $w_2$ satisfying 
\eqref{pex.defw1w2c} and \eqref{pex.defw}. 

\paragraph*{Iterative step.}

We consider Problem \ref{in.probpisc} for $d \geq 2$, $L = 2$, with the unknown potential $v$ satisfying \eqref{in.vprop} and with the background potentials $w_1$, $w_2$ satisfying \eqref{in.wjass}, \eqref{pex.defw1w2c}, \eqref{pex.defw}.

Let $v^*_E$ be an approximation to $v$ satisfying 
\eqref{eqs:itp.v*} for some $A > 0$, $\alpha \geq 0$, $K>0$ and $E^* = E^*(K,D_\text{ext})$, where $E^*(K,D_\text{ext})$ is defined according to \eqref{itp.E*def}, \eqref{isp.Dextdef} (with $\Omega_1 = \Omega_2$). 

We construct an improved approximation $v^{**}_E$ to the unknown potential 
$v$ via the scheme of Section \ref{sec.itp} with $\Phi\scatam$ replaced by 
$\scatapprox$ of formula \eqref{isp.fappr} of Section \ref{sec.isp}. Put
\begin{equation}\label{itc.U**def}
  \begin{gathered}
    U^{**}_E(p) = \scatapprox(p) - \Phi\scatamu{*}(p) + \widehat v^*_E(p), \\
    p \in \mathcal B^d_{2 \sqrt E}, \; \zeta^*(p,E) \neq 0,
    \end{gathered}
\end{equation}
where $\scatamu{*}$ is the scattering amplitude of $v^*_E$, and 
$\zeta^*$ is defined in \eqref{eq:zeta*}.

Under assumptions \eqref{pex.defw1w2c}, \eqref{pex.defw}, the iterative step for phaseless inverse scattering is realized as follows.

\begin{theorem}\label{itc.thmc} Let $v$ satisfy \eqref{in.vprop} and $v \in W^{n,1}(\mathbb R^d)$ for some $n>d$. Let $w_1$, $w_2$ be the same as in \eqref{in.wjass}, \eqref{pex.defw1w2c}, \eqref{pex.defw}. Let $v^*_E$ be an approximation to $v$ satisfying \eqref{eqs:itp.v*} for some $A>0$, $\alpha \geq 0$, $K > 0$ and for $E^* = E^*(K,D_\text{ext})$, where $E^*(K,D_\text{ext})$ is defined according to \eqref{itp.E*def}, \eqref{isp.Dextdef} (with $\Omega_1 = \Omega_2$). We suppose also that
\begin{equation}
  \alpha < \tfrac{n-d}{2(d+2\sigma)}, \label{itc.alphaas}
\end{equation}
where $\sigma$ is the constant of \eqref{pex.defw}. Let
\begin{equation}
  \begin{gathered}\label{itc.beta1def}
  v^{**}_E(x) := \begin{cases}
	\int_{\mathcal B^d_{r(E)}} e^{-ipx} U^{**}_E(p) \, dp,& x \in D,\\
    0, & x \not \in D,
		\end{cases}
		\qquad \mbox{where }
    r(E) := 2\tau E^{\frac{\beta}{n-d}} \mbox{ and } 
		\beta := \tfrac{(\frac 1 2 +\alpha)(n-d)}{n+2\sigma}
    \end{gathered}
\end{equation}
for some $\tau \in (0,1]$. 
Here $U^{**}_E$ is defined in \eqref{itc.U**def}, 
and $\mathcal B^d_r$ is defined by \eqref{pex.defBr}. 

Then  there exist constants $B_1 = B_1(\tau,\|v\|_{n,1},A,K,D_\text{ext},d,\sigma,n,\muc{sec.pex}{1})$ and 
$E_1 = E_1(\tau,\alpha,A,K,D,d,\sigma,n,\muc{sec.pex}{1})$ 
defined in \eqref{tcp.v**est} such that 
   \begin{equation}
 \begin{gathered}
    \|v^{**}_E-v\|_{L^{\infty}(D)} \leq B_1 E^{-\beta} \qquad 
		\mbox{for all }E \geq E_1.
 \end{gathered}
\end{equation}
\end{theorem}

\begin{remark} Under assumptions of Theorem \ref{itc.thmc}, $v^{**}_E(x)$ is well-defined for $E \geq E_1$, i.e.:
  \begin{equation}
	\begin{gathered}
	  \zeta^*(p,E) \neq 0 \quad \text{for $p \in \mathcal B^d_{r(E)}$, $E \geq E_1$},\\
	  r(E) \leq 2 \sqrt E \quad \text{for $E \geq E_1$}.
	\end{gathered}
 \end{equation}
\end{remark}

\begin{remark} The following three conditions are equivalent:
\begin{equation}\label{itc.alphaas2}
	\alpha < \tfrac{n-d}{2(d+2\sigma)}, \quad \beta < \tfrac{n-d}{2(d+2\sigma)}, \quad \alpha < \beta,
\end{equation}
where $\beta$ is defined in \eqref{itc.beta1def}, $n>d$, $\alpha>0$, and $\sigma>d$.
\end{remark}

Theorem \ref{itc.thmc} is proved in Section \ref{sec.tcp}.

Theorem \ref{itc.thmc} of the present work can be considered as an extension of Theorem~1 of \cite{Agal2016b} to the case when $v^* \neq 0$. However, Theorem \ref{itc.thmc} of the present work for $v^* = 0$ does not coincide with Theorem 1 of \cite{Agal2016b}.

\paragraph*{Iterations}
Let $v$, $w_1$, $w_2$ satisfy the assumptions of Theorem \ref{itc.thmc}. Let $u^{(1)}_E = v^{**}_E$ for $v^* = 0$. Note that $u^{(1)}_E$ is similar but does not coincide with the approximate reconstruction of Theorem 1 of \cite{Agal2016b}; see formulas \eqref{pex.ucerr}, \eqref{pex.defuc} of the present work. In particular, we have that
\begin{equation}
  \begin{gathered}
  \|u^{(1)}_E - v\|_{L^{\infty}(D)} =  \mathcal{O}(E^{-\alpha_1}), \quad 
	E \to + \infty \mbox{ with }
  \alpha_1 := \tfrac 1 2 \tfrac{n-d}{n+2\sigma}.
  \end{gathered}
\end{equation}
Then, applying the iterative step described above in this subsection  we construct nonlinear approximate reconstructions $u^{(j)}_E$, $j \geq 2$, such that
\begin{equation}\label{itc.ujconv}
 \begin{gathered}
	\|u^{(j)}_E - v\|_{L^\infty(D)} = \mathcal{O}(E^{-\alpha_j}), 
	\qquad E \to + \infty \mbox{ with }
	\alpha_j := \tfrac 1 2 \tfrac{n-d}{d+2\sigma} \bigl(1 - \bigl(\tfrac{n-d}{n+2\sigma}\bigr)^j \bigr).
 \end{gathered}
\end{equation}
The approximations $u^{(j)}_E$ of \eqref{itc.ujconv} for phaseless inverse scattering under assumptions \eqref{pex.defw1w2c}, \eqref{pex.defw} are analogs of approximations $u^{(j)}_E$ of \eqref{in.ujtov} for  phased inverse scattering. In a similar way with \eqref{in.alphalim}, 
\begin{equation}\label{itc.alphalim}
 \begin{array}{ll}
  \alpha_j \to \alpha_\infty = \tfrac 1 2 \tfrac{n-d}{d+2\sigma} & \text{as $j\to+\infty$},\\
  \alpha_j \to \tfrac j 2 & \text{as $n \to + \infty$},\\
  \alpha_\infty \to + \infty & \text{as $n \to + \infty$},
  \end{array}
\end{equation}
so that the convergence in \eqref{itc.ujconv} as $E \to + \infty$ is much more optimal then the convergence in \eqref{pex.ucerr}, at least, for large $j$, $n$.

\subsection{Iterations for background potentials of type B}\label{sec.itr}
In this subsection 
we consider Problem \ref{in.probpisc} for $d \geq 2$ with shifted background potentials $w_1$, $w_2$ 
described by \eqref{pex.defw}, and \eqref{pex.defw1w2r} and unknown potential $v$ satisfying 
\eqref{in.vprop}. 

\paragraph*{Iterative step.}
We consider the set $Z^\varepsilon_y$ of formula \eqref{pex.defZe}. Put \begin{equation}\label{itr.defpzt}
  \begin{gathered}
  p^\varepsilon(p_\bot,z,t) = p_\bot + \pi z \tfrac{y}{|y|^2} + t \varepsilon \tfrac{y}{|y|^2}, \\
  p_\bot \in \mathbb R^d, \; p_\bot \cdot y = 0, \quad z \in \mathbb Z, \quad t \in \mathbb R.
  \end{gathered}
\end{equation}
Note that
\begin{equation}
  \begin{gathered}
  \text{for any $p \in Z^\varepsilon_y$ there exists the unique triple $(p_\bot,z,t)$ such that}\\
  p^\varepsilon(p_\bot,z,t)  = p, \; p_\bot \in \mathbb R^d, \; p_\bot \cdot y = 0, \; z \in \mathbb Z, \; t \in (-1,1).
  \end{gathered}
\end{equation}
In addition to $U^{**}$ of \eqref{itc.U**def}, we also define
\begin{equation}\label{itr.defU**e}
    U^{**}_{N,\varepsilon}(p^\varepsilon(p_\bot,z,t),E) = \sum_{-N \leq j \leq N, j \neq 0} U^{**}\bigl(p^\varepsilon(p_\bot,z,j\bigr),E) L_{j}(t),
\end{equation}
under the assumptions that
\begin{equation}\label{itr.U**eass}
    \varepsilon < \tfrac{\pi}{N+1}, \; |p^\varepsilon(p_\bot,z,t)| \leq 2 \sqrt E - \tfrac{\pi}{|y|}, \; t \in (-1,1), \; \varepsilon > 0, \; N \geq 1,
\end{equation}
where
\begin{gather}
  L_j(t) = \frac{(-1)^{N-j} j (t+j)}{(N-j)!(N+j)!} \prod_{1\leq i \leq N, i \neq j} (t^2 - i^2), \quad j = \pm 1, \dots, \pm N. \label{itr.deflj}
\end{gather}
Note that for fixed $p_\bot \in \mathbb R^d$, $p_\bot \cdot y = 0$, and for fixed $z \in \mathbb Z$, function $U^{**}_{N,\varepsilon}(p^\varepsilon(p_\bot,z,t),E)$ is the Lagrange interpolating polynomial of degree $2N-1$ in $t \in \mathbb R$ for $U^{**}(p^\varepsilon(p_\bot,z,t),E)$ with the nodes at $t = \pm 1$, \dots, $\pm N$. In addition, $L_j(t)$ is the $j$-th elementary Lagrange interpolating polynomial of degree $2N-1$:
\begin{equation}
 L_j(t) = \begin{cases} 1, & t = j, \\ 0, & t = \pm 1, \dots, \pm N, \; t \neq j. \end{cases}
\end{equation}

Note also that if assumptions \eqref{itr.U**eass} are valid for some $p_\bot \in \mathbb R^d$, $p_\bot \cdot y = 0$, $z \in \mathbb Z$, $t \in (-1,1)$, then
\begin{equation}\label{itr.pjinset}
	p^\varepsilon(p_\bot,z,s) \in \mathcal B^d_{2\sqrt E} \setminus Z^\varepsilon_y \quad
    \text{for all $s \in [-N,-1] \cup [1,N]$}.
\end{equation}

Under assumptions \eqref{pex.defw},  \eqref{pex.defw1w2r}, the iterative step is realized as follows.

\begin{theorem}\label{itr.thmr} Let $v$ satisfy \eqref{in.vprop} and $v \in W^{n,1}(\mathbb R^d)$ for some $n>d$. Let $w_1$, $w_2$ be the same as in \eqref{in.wjass}, \eqref{pex.defw}, \eqref{pex.defw1w2r}. Let  $v^*_E$ be an approximation to $v$ satisfying \eqref{itp.v*prop}, \eqref{eqs:itp.v*} for some $A > 0$, $\alpha \geq 0$, $K > 0$ and for $E^* = E^*(K,D_\text{ext})$, where $E^*(K,D_\text{ext})$ is defined according to \eqref{itp.E*def}, \eqref{isp.Dextdef} (with $\Omega_1 = \Omega_2$). We suppose also that
\begin{equation}
  \alpha < \tfrac{n-d}{2(d+2\sigma + \frac{n-d}{2N+1})} \quad \text{for some $N \geq 1$}, \label{itr.alphaas}
\end{equation}
where $\sigma$ is the constant of \eqref{pex.defw}. Let
\begin{equation}\label{itr.beta2def}
  \begin{aligned}
& v^{**}_E(x) := \begin{cases}  \int\limits_{ \mathcal B^d_{r(E)} \setminus Z^{\varepsilon(E)}_y} e^{-ipx} U^{**}_E(p) \, dp +,\int\limits_{ \mathcal B^d_{r(E)} \cap Z^{\varepsilon(E)}_y} e^{-ipx}  U^{**}_{N,\varepsilon(E)}(p,E) \, dp &x\in D  \\
	0, &x\notin D \end{cases}\\
& \mbox{with }  r(E) = 2\tau E^{\frac{\beta}{n-d}}, \quad \varepsilon(E) = E^{-\frac{\beta}{2N+1}}, \quad 
   \beta = {\tfrac{(\frac 1 2 +\alpha)(n-d)}{n+2\sigma+\tfrac{n-d}{2N+1}}}, \quad \text{for some $\tau \in (0,1]$},
  \end{aligned}
\end{equation}
where $U^{**}_E$ is defined by formulas \eqref{itc.U**def}; 
$U^{**}_{N,\varepsilon}(p,E)$ is defined by \eqref{itr.defU**e}; $\mathcal B^d_r$, $Z^\varepsilon_y$ are defined by \eqref{pex.defBr}, \eqref{pex.defZe}. 

Then the following estimate holds:
\begin{equation}\label{itr.v**err}
  \begin{gathered}
    \|v^{**}_E-v\|_{L^\infty(D)} \leq B_2 E^{-\beta},  \quad\; E \geq E_2,
  \end{gathered}
\end{equation}
  where 
  $B_2 = B_2(\tau,y,N,\|v\|_{n,1},A,K,D_\text{ext},d,\sigma,n,\muc{sec.pex}{1})$ and
  $E_2 = E_2(\tau,y,\alpha,N,A,K,D,d,\sigma,n,\muc{sec.pex}{1})$ 
are defined in \eqref{trp.B2E2def}.

\end{theorem}

Theorem \ref{itr.thmr} is proved in Section \ref{sec.trp}.

\begin{remark} Under assumptions of Theorem \ref{itr.thmr}, $v_{E,1}^{**}(x)$ and $v_{E,2}^{**}(x)$ are well-defined for $E \geq E_2$, i.e.:
\begin{equation}
 \begin{gathered}
	\zeta^*(p,E) \neq 0 \quad \text{for $p \in \mathcal B^d_{r(E)} \setminus Z^{\varepsilon(E)}_y$, $\quad E \geq E_2$},\\
	r(E) \leq 2\sqrt E \quad \text{for $E \geq E_2$},
 \end{gathered}
\end{equation}
where $\zeta^*$ is defined by \eqref{eq:zeta*}, and also
\begin{equation}
   \varepsilon(E) < \tfrac{\pi}{N+1}, \quad |p| \leq 2 \sqrt E - \tfrac{\pi}{|y|} \quad \text{for $p \in \mathcal B^d_{r(E)}$, $E \geq E_2$}.
\end{equation}
\end{remark}

\begin{remark} The following three conditions are equivalent:
\begin{equation}\label{itr.alphaas2}
	\alpha < \tfrac{n-d}{2(d+2\sigma + \frac{n-d}{4N+1})}, \quad \beta < \tfrac{n-d}{2(d+2\sigma+\frac{n-d}{4N+1})}, \quad \alpha < \beta,
\end{equation}
where $\beta$ is defined in \eqref{itr.beta2def}, $n>d$, $\alpha>0$, and $\sigma>d$. 
In addition, each of conditions \eqref{itr.alphaas2} is equivalent to the following pair of conditions: 
\begin{equation}
  \alpha < \tfrac{n-d}{2(d+2\sigma)}, \quad  N > \tfrac 1 2 \left( \alpha^{-1} - \bigl( \tfrac{n-d}{2(d+2\sigma)} \bigr)^{-1} \right)^{-1} - \tfrac 1 4.
\end{equation}
\end{remark}

\paragraph*{Iterations.}
Let $v$, $w_1$, $w_2$ satisfy the assumptions of Theorem \ref{itr.thmr}. Let 
$u^{(1)}_E = v^{**}_E$ for $v^* = 0$. Note that $u^{(1)}_E$ is similar, but does not coincide with the approximate reconstruction of Theorem 2 of \cite{Agal2016b}; see formulas \eqref{pex.usherr}, \eqref{pex.defush} of the present article. In particular, we have that
\begin{equation}
  \begin{gathered}
  \|u^{(1)}_E- v\|_{L^\infty(D)} + \mathcal{O}(E^{-\alpha_1}),\qquad  E \to + \infty\quad \mbox{with }
  \alpha_1 = \tfrac 1 2 \tfrac{n-d}{n+2\sigma+\tfrac{n-d}{2N+1}}.
  \end{gathered}
\end{equation}
Then, applying the iterative step described above in this subsection we construct nonlinear approximate reconstructions $u_E^{(j)}$, $j \geq 2$, such that
\begin{equation}\label{itr.ujconv}
 \begin{gathered}
	\|u_E^j - v\|_{L^\infty(D)} 
	=  \mathcal{O}(E^{-\alpha_j}) \quad  E \to + \infty \quad \mbox{with }
	\alpha_j = \tfrac 1 2 \tfrac{n-d}{d+2\sigma+\tfrac{n-d}{2N+1}} \biggl(1 - \biggl(\tfrac{n-d}{n+2\sigma+\tfrac{n-d}{2N+1}}\biggr)^j \biggr).
 \end{gathered}
\end{equation}
These approximations $u^{(j)}$ for phaseless inverse scattering under assumptions \eqref{pex.defw}, \eqref{pex.defw1w2r} are analogs of approximations $u^{(j)}$ of \eqref{in.ujtov} for phased inverse scattering. In a similar way with \eqref{in.alphalim} and \eqref{itc.alphalim},
\begin{equation}
 \begin{array}{ll}
  \alpha_j \to \alpha_\infty = \tfrac 1 2 \tfrac{n-d}{d+2\sigma+\tfrac{n-d}{2N+1}} & \text{as $j\to+\infty$},\\
  \alpha_j \to \tfrac j 2 & \text{as $n \to + \infty$, $N\to+\infty$},\\
  \alpha_\infty \to + \infty & \text{as $n \to + \infty$},
  \end{array}
\end{equation}
so that the convergence in \eqref{itr.ujconv} as $E \to + \infty$ is much faster then the convergence in \eqref{pex.usherr}, at least, for large $j$, $n$, $N$.

\section{Numerical experiments}\label{sec.num}

\subsection{Implementation of the Fourier transform and its inverse} \label{sec.imp.ft}
The iterative algorithm presented in sections \ref{sec.itc}, \ref{sec.itr} is implemented in {\tt Matlab} in the two-dimensional case. In our implementation we represent potentials $v$ and $w_l$ by discrete functions $\underline{v}$, $\underline{w}_l$ defined on the space-variable grid
\begin{equation}\label{eq:Gamma_s}
\begin{aligned}
&\Gamma^{\rm s}:=\left\{\tfrac{2}{\spcBd}(\spcind_1,\spcind_2):\spcind_1,\spcind_2\in\mathbb{Z}_\spcBd\right\} 
\mbox{ for }\spcBd\in 2\mathbb{N}, \spcBd\geq \tfrac{2\sqrt{E}}{\pi} \\
&\text{where }\mathbb Z_\spcBd := \bigl\{ -\tfrac{\spcBd}{2}, -\tfrac{\spcBd}{2} + 1, \dots, \tfrac{\spcBd}{2} - 1 \bigr\}. 
\end{aligned}
\end{equation}
In turn, the input data $|f(k,l)|^2$, $|f_l(k,l)|^2$ are measured on a grid
\begin{equation*}
\mathcal{M}_E^{\rm m}=\{(k_m,l_m):m=1..M\}\subset \mathcal{M}_E,
\end{equation*}
whose precise form depends on the experimental setup. This leads to 
the following grid in Fourier space:
\begin{equation*}
	\Gamma^{\rm m}:= \{k_m-l_m:m=1..M\}
\end{equation*}
\paragraph*{Minimal data.} To approximate the inverse Fourier transform of a function, which is supported on $\mathcal B^2_{2\sqrt{E}}$ and sampled on  
the grid $\Gamma^{\rm m}$, by the Fast Fourier transform (FFT), 
$\Gamma^{\rm m}$ has to be rectangular. If $\Gamma^{\rm s}$ is given 
by \eqref{eq:Gamma_s}, a minimal choice of $\Gamma^{\rm m}$ is 
\begin{equation}\label{eq:Gamma_min}
\Gamma^{\rm m}_{\rm min} := \mathcal{B}^2_{2\sqrt{E}}\cap \pi\mathbb{Z}^2.
\end{equation}
A corresponding measurement grid $\mathcal{M}_{E,\rm min}^{\rm m}$ can 
be defined as in \eqref{in.defGE} replacing 
$\mathcal{B}^2_{2\sqrt{E}}$ by $\Gamma^{\rm m}_{\rm min}$. 
Extending a function given on $\Gamma^{\rm m}$ to the exterior grid 
\begin{equation}\label{eq:Gamma_e}
	\Gamma^{\rm e}:= \{\pi(\spcind_1,\spcind_2): \spcind_1,\spcind_2\in\mathbb{Z}_\spcBd,
	\pi^2(\spcind_1^2+\spcind_2^2)>4E\}
\end{equation}
by $0$, we can compute an approximation to the inverse Fourier transform 
on $\Gamma^{\rm s}$ by FFT. 

\paragraph{Discrete Ewald circles.}
The above choice $\Gamma^{\rm m}_{\rm min}$ of the set of measurement points 
is inconvenient both from an experimental and from a computational point 
of view since each point in $\Gamma^{\rm m}_{\rm min}$ corresponds to 
a different incident wave. If a scattering experiment is performed for 
some incident wave or if the solution to the scattering problem is computed 
numerically, the resulting far field pattern can be evaluated at other points 
without essential additional costs. 

Therefore, we now consider input data for $M_1$ 
uniformly distributed incident wave vectors $k$ where each far field pattern is evaluated 
at $M_2$ uniformly distributed scattered wave vectors $l$:
\begin{equation}\label{imp.kl}
	\begin{gathered}
	k(s) := \sqrt E \bigl( \cos(2\pi \tfrac{s}{M_1}), \sin(2 \pi \tfrac{s}{M_1} ) \bigr),\\
	l(s,t) := \sqrt E \bigl( \cos( \tfrac{2\pi s}{M_1} + \tfrac{2\pi t}{M_2} ), \sin( \tfrac{2\pi s}{M_1} + \tfrac{2\pi t}{M_2} ) \bigr),
	\end{gathered}\quad s \in \mathbb Z_{M_1}, \; t \in \mathbb Z_{M_2},
\end{equation}
resulting in 
\begin{align}\label{imp.measgrid}
&	\mathcal M_{E,M_1,M_2}^{\rm m} := \bigl\{ (k(s),l(s,t)) \mid s\in\mathbb Z_{M_1},
	t\in\mathbb{Z}_{M_2}\bigr\}, \quad M = M_1 M_2,\\
&	\Gamma^{\rm m}_{M_1,M_2}:= \{k-l\colon (k,l)\in \mathcal M_{E,M_1,M_2}^{\rm m}\}.
\end{align}
This is illustrated in Fig.~\ref{fig.ES}: The points of 
$\Gamma^{\rm m}=\Gamma^{\rm m}_{M_1,M_2}$ corresponding to a given incident wave vector $k$ are located on the red circle passing through the origin $O$ and centered at point $O'$ such that $\overrightarrow{OO'} = k$. These points of $\Gamma^{\rm m}$ corresponding to a fixed $k$ are also called (discrete) Ewald circle in the physical literature. 
%

\begin{figure}[ht]
\begin{tabular}{ccc}
 \includegraphics[width=.3\linewidth]{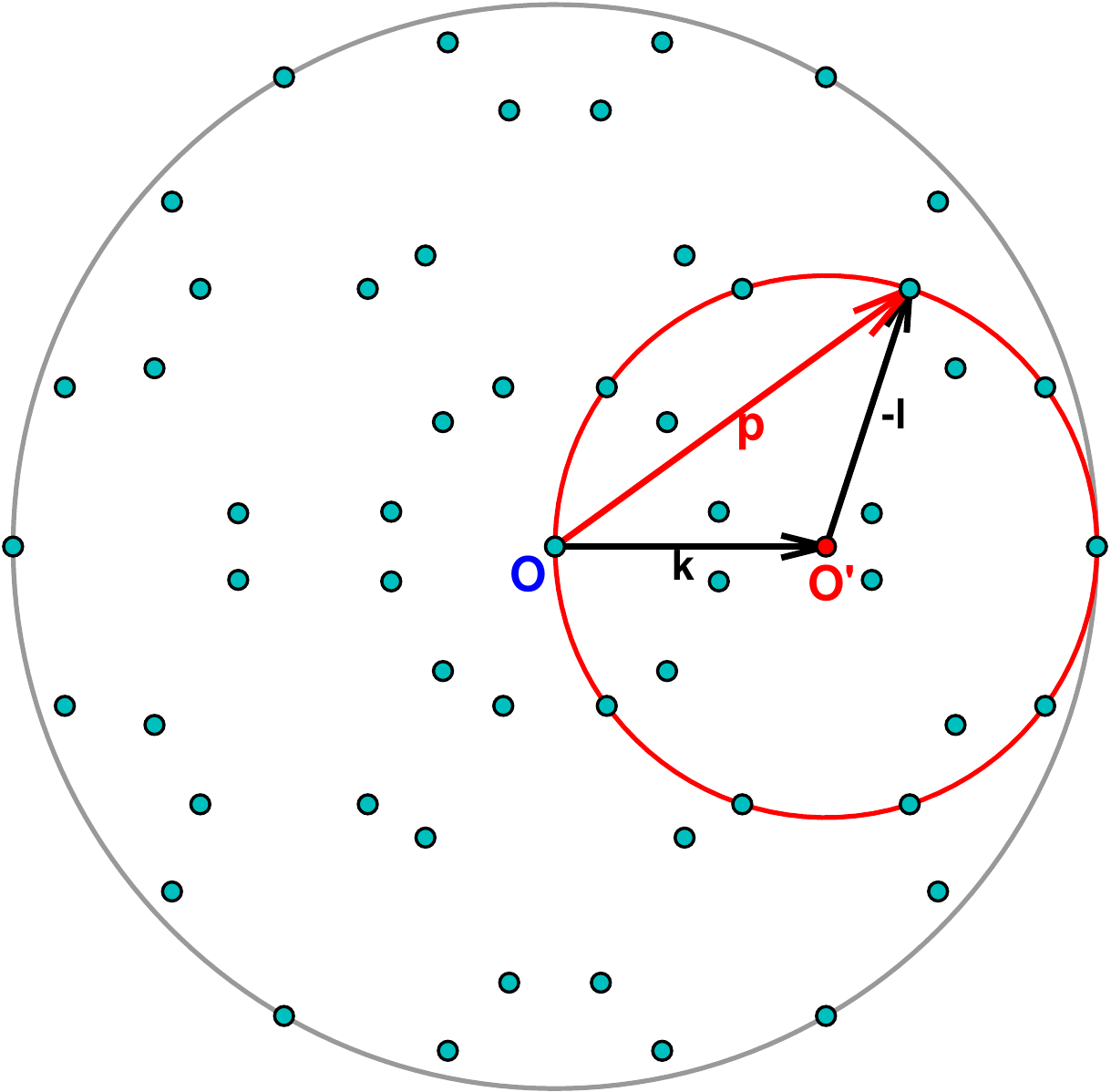}
& \includegraphics[width=.3\linewidth]{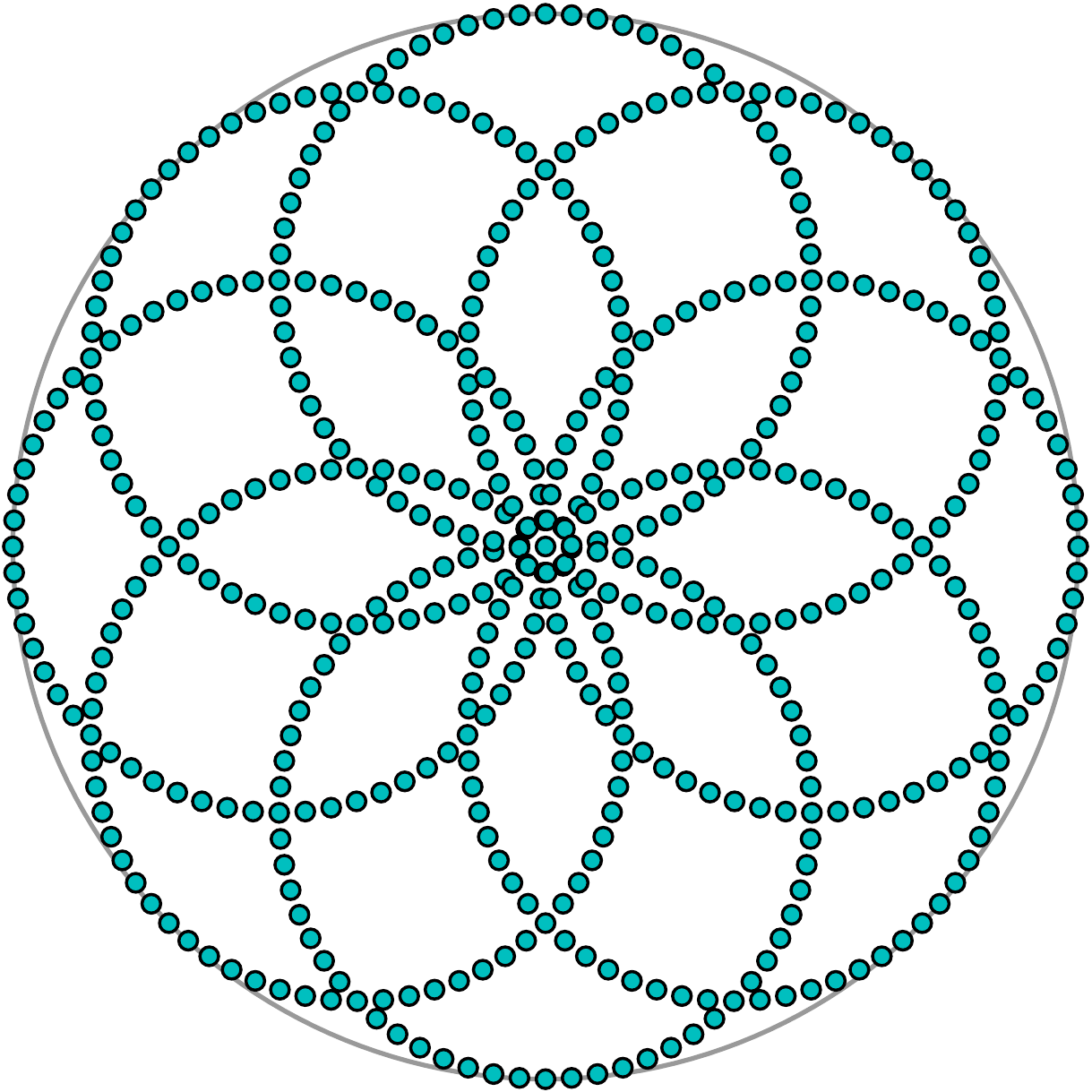}
& \includegraphics[width=.3\linewidth]{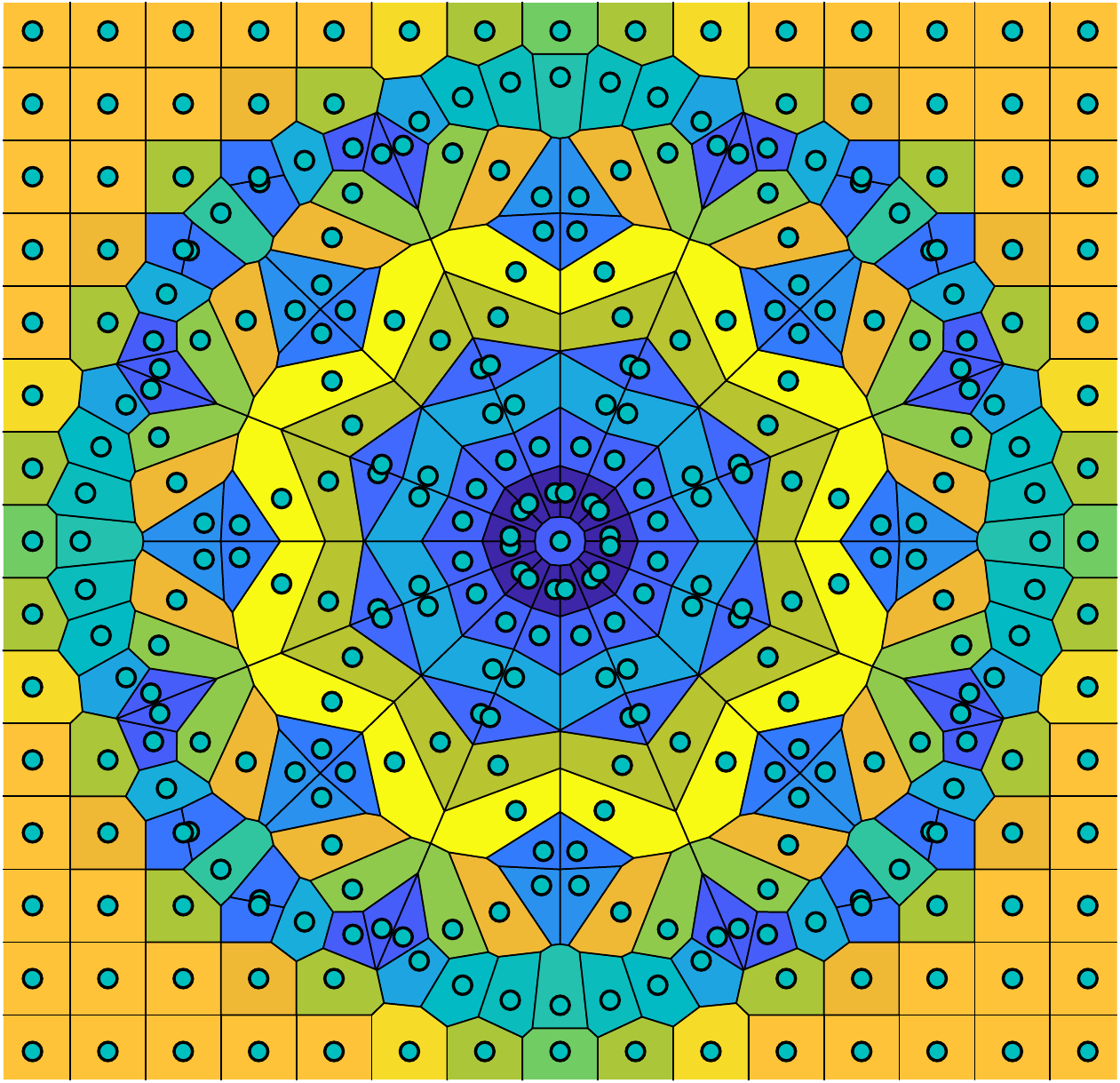}\\
  $M_1 = 6$, $M_2 = 10$
&  $M_1=8$, $M_2 = 64$
&  $M_1 = 8$, $M_2 = 30$
\end{tabular}
\caption{Illustration of sets of transformed measurement points 
$\Gamma^{\rm m}_{M_1,M_2}$ on 
Discrete Ewald circles for different values of the numbers $M_1$ of incident waves 
and far field points $M_2$. The right panel shows 
the Voronoi diagram corresponding to the grid $\Gamma^{\rm e} \cup \Gamma^{\rm m}_{M_1,M_2}$. The color is related to the area of a cell.}
\label{fig.ES}
\end{figure}

Given a discrete function $\underline v \colon \Gamma^s \to \mathbb C$ representing a function $v \colon \mathbb R^2 \to \mathbb C$, the Fourier transform of $v$ can be approximately represented by a discrete function $\underline{\widehat v}\colon \Gamma^m \cup \Gamma^e \to \mathbb C$ such that $\underline{\widehat v} = A \underline v$, where 
\[
A:=\left[(\spcBd\pi)^{-2}\exp(ix\cdot p)\right]_{p\in \Gamma^{\rm m}\cup \Gamma^{\rm e},x\in\Gamma^{\rm s}}.
\] 
Here it is necessary to include the points in $\Gamma^{\rm e}$ to obtain 
small condition numbers of $A$ since the inverse Fourier transform is 
computed by numerically  inverting $A$. Matrix-vector products with 
$A$ and $A^*$ can be computed efficiently without the need to set 
up and store the matrix $A$ using the 
%
Nonequispaced Fast Fourier Transform (NFFT). In our work we use the NFFT implementation of \cite{keiner2009using}. 
The definitions of the grids $\Gamma^{\rm s}$, $\Gamma^{\rm m}$, $\Gamma^{\rm e}$ and of the Fourier transform matrix $A$ are summarized in Algorithm \ref{alg:Fourier}.

\begin{algorithm}[ht]
{\bf data:} spatial grid: 
$\Gamma^{\rm s}:=\{\frac{2}{\spcBd}(\spcind_1,\spcind_2):\spcind_1,\spcind_2\in\mathbb{Z},-\frac{\spcBd}{2} 
\leq \spcind_1,\spcind_2<\frac{\spcBd}{2}\}$ with $\spcBd\in 2\mathbb{N}$, $\spcBd\geq \frac{2\sqrt{E}}{\pi}$ \\
measurement points: $\mathcal{M}_E^{\rm m}=\{(k_m,l_m):m=1..M\}\subset \mathcal{M}_E$\\ 
{\bf results:} 
Fourier space grids inside and outside $\mathcal{B}^2_{2\sqrt{E}}$: 
$\Gamma^{\rm m}$, $\Gamma^{\rm e}$\\
matrix representing Fourier transform: $A$\\
pushforward matrix from the measurement grid $\mathcal M^m_E$ to the Fourier space grid $\Gamma^m$: $\Phi$ 
\smallskip

\begin{algorithmic}
\State $\Gamma^{\rm m}:= \{k_m-l_m:m=1..M\}$; 
\Comment{grid inside $\mathcal{B}^2_{2\sqrt{E}}$}
\State $\Gamma^{\rm e}:= \{\pi(\spcind_1,\spcind_2): \spcind_1,\spcind_2\in\mathbb{Z},-\frac{\spcBd}{2} 
\leq \spcind_1,\spcind_2<\frac{\spcBd}{2}, \pi^2(\spcind_1^2+\spcind_2^2)>4E\}$;
\Comment{grid outside $\mathcal{B}^2_{2\sqrt{E}}$}
\State $A:=\left[(\pi \spcBd)^{-2}\exp(ix\cdot p)\right]_{p\in \Gamma^{\rm m}\cup \Gamma^{\rm e},x\in\Gamma^{\rm s}}$;  \Comment{avoid computation of $A$ and use {\rm NFFT} instead}
\State $(\Phi\underline{f})_p := \frac{1}{\#(M_p)} 
\sum_{m\in M_p} \underline f_m$ where $M_p:= \{m : k_m-l_m=p\}$ 
\end{algorithmic}
\caption{function [$\Gamma^{\rm m}, \Gamma^{\rm e}, A, \Phi] = 
\texttt{fourier\_setup}(\Gamma^{\rm s},\mathcal{M}_E^{\rm m})$}\label{alg:Fourier}
\end{algorithm}

\paragraph{Approximation of the inverse Fourier transform.}
A first idea may be to approximate the continuous inverse Fourier transform 
by the Moore-Penrose inverse of $A$. 
If $A$ is injective, this corresponds to the solution of a least-squares problem 
\[
A^{\dagger}\underline{U} = \argmin_{\underline{w}}\|A\underline{w}-\underline{U}\|^2_2
\]
which can be achieved by the conjugate gradient (CG) method applied 
to the normal equation $(A^*A)\underline{w}=A^*\underline{U}$. The ideal 
situation would be that $A$ is isometric, i.e.\ 
$\|A\underline{w}\|_2=\|\underline{w}\|_2$ for all 
$\underline{w}:\Gamma^{\rm s}\to \mathbb{C}$. In this case 
$A^{\dagger}=A^*$, and the CG method would 
yield the exact solution in the first step. However, in our situation $A$ is typically far from being isometric, so that $A^* A$ is far from the identity matrix, and the CG method requires a big number of iterations. The reason is that even though the continuous Fourier transform is isometric, the Euclidean norm $\|\underline{U}\|_2$ of the sampled version 
$\underline{U}:\Gamma^{\rm m}\cap\Gamma^{\rm e}\to\mathbb{C}$ of 
a function $U:[-\tfrac{\spcBd\pi}{2},\tfrac{\spcBd\pi}{2}]^2\to\mathbb{C}$
can be a bad approximation for $\|U\|_{L^2}$.

To overcome this difficulty, we design a weight matrix $D$ such that $\|D^{1/2}\underline{U}|\|_2\approx \|U\|_{L^2}$. 
Then we approximate the inverse Fourier transform by the Moore-Penrose 
inverse of $A$ with respect to this weighted norm 
$\|D^{1/2}\cdot\|_2$, or in matrix notation
\[
(D^{1/2}A)^{\dagger}D^{1/2}\underline{U}= \argmin_{\underline{w}}\|D^{1/2}A\underline{w}-D^{1/2}\underline{U}\|^2_2.
\]
Recall that by the first-order optimality conditions, which are necessary and sufficient for convex functionals, this minimization problem is equivalent to solving the normal equation $A^* D A \underline w = A^* D \underline U$.

To construct a matrix $D$ such that $\|D^{1/2}\underline{U}|\|_2^2
\approx \int |U(p)|^2\,dp$, we use a Voronoi partition of the square 
$[-\tfrac{\pi \spcBd} 2, \tfrac{\pi \spcBd} 2]^2$ into cells $C(p)$ centered at points 
$p \in \Gamma^m \cup \Gamma^e$. In {\tt Matlab} this subdivision is computed by the built-in function {\tt voronoi}, see Fig.~\ref{fig.ES}. To approximate 
the integral by a Riemann sum we evaluate the area $|C(p)|$ of each cell 
$C(p)$, $p\in\Gamma^{\rm m}\cup\Gamma^{\rm e}$ and choose $D$ as the 
diagonal matrix $D:=\diag(|C(p)|)_{p\in\Gamma^{\rm m}\cup\Gamma^{\rm e}}$. 
The use of this matrix $D$ drastically decreases the number of CG steps and 
allows the approximate evaluation of the inverse Fourier transform 
with just a few CG steps. The reason is that $A^* D A$ is much closer to the identity matrix as $A^* A$.

\begin{remark}
	%
	If the function $\underline U$ is defined only on some subgrid $G \subset \Gamma^m \cup \Gamma^e$, two minor modifications are necessary: (i) The Fourier matrix $A$ must be restricted to the grid $G$ yielding matrix $\tilde A = [A_{pq}]_{p\in\Gamma,q\in\Gamma^s}$; (ii) The matrix of Voronoi weights $D$ must be computed for the grid $G$.
\end{remark}

\subsection{Implementation of the inversion method}\label{sec.imp.inv}

\paragraph{Phaseless Born approximation.}

Recall that in our implementation the potentials $w_l$ are represented by discrete functions 
$\underline {w}_l \colon \Gamma^s \to \mathbb C$, and the measured phaseless farfield data $|f|^2$, $|f_l|^2$  are represented by the discrete functions $\underline F$ and $\underline F_l \colon \mathcal M^{\rm m}_E \to \mathbb R$. Also note that in addition to the background potentials and phaseless farfield data, a cutoff radius $r_1>0$ and a threshold $\delta > 0$ are specified as input data for the algorithm. The cutoff radius $0<r_1\leq 2\sqrt E$ is analogous to the radius $r(E)$ of formulas \eqref{pex.defuc}, \eqref{pex.defush}, whereas the threshold $\delta$ is analogous to threshold $\varepsilon(E)$ of formula \eqref{pex.defush}.

The implementation of the phaseless Born approximation is shown in Algorithm \ref{alg.born}. 
The algorithm is formulated for an arbitrary number $L\geq 2$ of reference potentials, but reduces 
to the algorithm in the theoretical part of this paper if $L=2$. 
The principal part is the computation of the reduced grid $\tilde \Gamma^m \subset \Gamma^m$, which consists of points $p \in \Gamma^m$ meeting the threshold and the cutoff constraints, and of the function $\underline{U} \colon \tilde\Gamma^m \cup \Gamma^e \to \mathbb C$, which is the discrete version of the function $U_{\widehat w_1,\widehat w_2}$ of Subsection \ref{sec.pex}. This computation starts by initializing $\tilde \Gamma^m$ by the empty grid and $\underline U$ by the zero function and proceeds as follows:
\begin{enumerate}
	\item For each point $p \in \Gamma^m$ and each pair $w_{\tilde{l}}, w_l$ of reference potentials, 
	compute the determinant $\underline \zeta_p^{\tilde{l},l}$ corresponding to 
	$\zeta_{\widehat w_1,\widehat w_2}(p)$ of formula \eqref{pex.defzeta}. Since the approximate 
	Fourier transform $\underline{U}_p$ of $\underline{v}$ can be computed from any pair $(\tilde{l},l)$ 
	for which $\underline \zeta_p^{\tilde{l},l}\neq 0$ and since the computation is the more stable the 
	larger $|\underline \zeta_p^{\tilde{l},l}|$, we choose the pair $(\tilde{l}(p),l(p))$, for which 
	$|\underline \zeta_p^{\tilde{l},l}|$ is largest. 
	\item If $|\zeta_p^{\tilde{l}(p),l(p)}| > \delta$ (threshold constraint) include $p$ in $\tilde \Gamma^m$. 
	In addition, 
	if $|p|<r_1$ (cutoff constraint) compute $\underline{U}_p$ using Algorithm \ref{alg:Urec} with 
	appropriate parameters. 
	Rather than implementing an explicit interpolation scheme at points where $\zeta_p^{\tilde{l}(p),l(p)}$ 
	vanishes or is too small as done in \eqref{pex.Uedef} for our theoretical analysis, we use 
	a trigonometric interpolation of the discrete Fourier transform induced by fitting 
	to the remaining points of $\Gamma^{\rm m}$. This procedure is easier to implement 
	and allows the numerical treatment of arbitrary background potentials. 
\end{enumerate}

The last step is to compute the function $\underline{v}^*_E \colon \Gamma^s \to \mathbb C$, which is the discrete version of the phaseless Born approximation $u_E$ of Subsection \ref{sec.pex}, as the inverse Fourier transform of the function represented by $\underline{U}$. This step is explained in Subsection \ref{sec.imp.ft}.

%
%
%
%
%
%

\begin{algorithm}[ht]
{\bf data:} 
measured data $(\underline{F}_0)_m\approx|f(k_m,l_m)|^2$, 
$(\underline{F}_l)_m\approx|f_l(k_m,l_m)|^2$ 
 for $m=1..M$, $l=1..L$\\
(discrete) background potentials 
$\underline{w}_1, \dots,\underline{w}_L:\Gamma^{\rm s}\to \mathbb{C}$\\
$\Gamma^{\rm s}$, $\mathcal{M}_E^{\rm m}$ as in $\texttt{fourier\_setup}$\\
cutoff radius $r_1>0$  \\
threshold $\delta>0$

{\bf results:} potential reconstruction: {$\underline{v}^*_E$} \\
reduced Fourier space grid and Fourier transform matrix: $\widetilde{\Gamma}^{\rm m}$, $\widetilde{A}$\\
Voronoi diagonal weight matrix: $D$
\smallskip

\begin{algorithmic}
\State [$\Gamma^{\rm m}, \Gamma^{\rm e}, A, \underline{\Phi}] = 
\texttt{fourier\_setup}(\Gamma^{\rm s},\mathcal{M}_E^{\rm m})$;
\Comment{see Algorithm \ref{alg:Fourier}}
\State		$\widehat{\underline{w}_l} := A\, \underline{w}_l\quad\mbox{for }l=1,\dots,L$;
\Comment{use {\rm NFFT}}
\State $\mathcal{P}:=\{(\tilde{l},l)\in \{1,..,L\}^2\colon \tilde l <l\}$;
\State $\tilde{\Gamma}^{\rm m}:=[\,]$; \quad
$\underline{U}_p:=0$ for $p\in \Gamma^{\rm m}\cup \Gamma^{\rm e}$;
\For {$p\in \Gamma^{\rm m}$}
\For {$(\tilde{l},l)\in \mathcal{P}$}
\State $\underline \zeta^{\tilde{l},l}_p = \Re \underline{\widehat w}_{\tilde{l},p} \Im \underline{\widehat w}_{l,p} - \Im \underline{\widehat w}_{\tilde{l},p} \Re \underline{\widehat w}_{l,p}$;
\EndFor
\State Choose $(\tilde{l}(p),l(p))\in\operatorname{argmax}_{(\tilde{l},l)\in\mathcal{P}} |\zeta^{\tilde{l},l}_p|$;
\If{$|\underline{\zeta}^{\tilde{l}(p),l(p)}_p|>\delta$}
\State 
$\tilde{\Gamma}^{\rm m}:=\tilde{\Gamma}^{\rm m}\cup\{p\}$;
\If{$|p|<r_1$}
\State $\underline{U}_p:= \texttt{Urec}\left({\underline{\widehat w}_{\tilde{l}(p),p}},\underline{\widehat w}_{l(p),p},
(\Phi \underline{F}_0)_p, (\Phi\underline{F}_{\tilde{l}(p)})_p,
(\Phi \underline{F}_{l(p)})_p\right)$; \Comment{see Algorithm \ref{alg:Urec}}
\EndIf
\EndIf
\EndFor
\State 
$\tilde{A}:=\left[A_{px}\right]_{p\in \tilde{\Gamma}^{\rm m}\cup \Gamma^{\rm e}
,x\in\Gamma^{\rm s}}$; 
\State 
$D:= \texttt{voronoi\_weights}\left(\tilde{\Gamma}^{\rm m}\cup \Gamma^{\rm e}\right)$;
\State
$\underline{v}^*_E := (D^{1/2}\tilde{A})^{\dagger}D^{1/2}[\underline{U}_{p}]
_{p\in \tilde{\Gamma}^{\rm m}\cup \Gamma^{\rm e}}$; \Comment{may be computed by CG and NFFT}
\end{algorithmic}
		\caption{function $[v^*_E,\widetilde{\Gamma}^{\rm m},\tilde{A},D]= \texttt{phaseless\_born\_inv}\left(
\underline{F}_0,\underline{F}_1,\underline{w}_1,\dots, \underline{F}_L,\underline{w}_L,
\Gamma^{\rm s},\mathcal{M}^{\rm m}_E,r_1,\delta\right)$}\label{alg.born}
\end{algorithm}

\paragraph{Iterative algorithm.}

In addition to the input parameters of the phaseless Born approximation, 
the iterative algorithm requires 
cutoff radii $0< r_1 \leq \cdots \leq r_J \leq 2\sqrt E$ to be specified. The implementation of the iterative method is shown in Algorithm \ref{alg.iter}.

The first step is the computation of the phaseless Born approximation $\underline v^*_E \colon \Gamma^s \to \mathbb C$, as well as of the grids $\Gamma^m$, $\Gamma^e$, $\tilde \Gamma^m$, Fourier matrices $A$, $\tilde A$ and Voronoi's  matrix of weights $D$, as explained above. The main part of the algorithm is the iteration procedure producing improved approximations~$\underline v^*_E$.

The iterative step starts by evaluating the scattering amplitudes $f^*_E$ and $f^*_{E,l}$, $l=1,..L$ 
of the potentials represented by discrete functions $\underline v^*_E$, $\underline v^*_E+\underline w_l$ at the points of the grid $\mathcal M^m_E$ yielding discrete functions 
$\underline f^*_E, \underline f^*_{E,l} \colon \mathcal M^m_E \to \mathbb C$. In principle, 
any black-box solver can be used to evaluate the scattering amplitudes. 
In our work we use the solver described in \cite{vainikko:00,hohage:01}.

The iterative step proceeds by computing the discrete function $\underline{\tilde f}^\text{appr}_E \colon \tilde\Gamma^m \cup \Gamma^e \to \mathbb C$, which is the discrete analog of function $\widetilde f^\text{appr}_E$ of \eqref{isp.fappr}, as well as the function $\underline U^{**} \colon \tilde\Gamma^m \cup\ \Gamma^e\to \mathbb C$, which is analogous to function $U^{**}_E$ of \eqref{itc.U**def}. This computation starts by initializing $\underline{\tilde f}^\text{appr}_E$ and $\underline U^{**}$ by the zero functions and continues as follows:
\begin{enumerate}
	\item For each point $p \in \tilde \Gamma^m$ such that $|p|<r_j$, where $j\geq 2$ is the current iteration number\footnote{by convention, $j=1$ corresponds to phaseless Born approximation}, compute $(\underline{\tilde f}^\text{appr}_E)_p$ using Algorithm \ref{alg:Urec} with appropriate parameters. If $L\geq 3$ 
	reference potentials are available, we again choose the most stable pair at each point $p$. 
	
	\item Evaluate $\underline{U}^{**}_p$ according to formula \eqref{itc.U**def}.
\end{enumerate}

The iteration ends by computing $\underline v^*_E$, as the inverse Fourier transform of the function represented by $\underline U^{**}$. The computation is explained in Subsection \ref{sec.imp.ft}.


\begin{algorithm}[ht]
{\bf data:} 
measured data $(\underline{F}_0)_m\approx|f(k_m,l_m)|^2$, 
$(\underline{F}_l)_m\approx|f_l(k_m,l_m)|^2$ 
 for $m=1..M$, $l=1..L$\\
(discrete) background potentials 
$\underline{w}_1,\dots, \underline{w}_L:\Gamma^{\rm s}\to \mathbb{C}$\\
$\Gamma^{\rm s}$, $\mathcal{M}_E^{\rm m}$ as in {\tt fourier\_setup}\\
cutoff radii $0<r_1\leq r_2\leq \dots\leq r_J\leq 2\sqrt{E}$  \\
threshold $\delta>0$

{\bf result:} potential reconstruction {$\underline{v}^*_E$} \smallskip

\begin{algorithmic}
\State $[\Gamma^{\rm m}, \Gamma^{\rm e}, A,\Phi] := 
\texttt{fourier\_setup}(\Gamma^{\rm s},\mathcal{M}_E^{\rm m})$;
\Comment{see Algorithm \ref{alg:Fourier}}
\State $[\underline v^*_E,\widetilde{\Gamma}^{\rm m},\tilde{A},D]:= \texttt{phaseless\_born\_inv}\left(
\underline{F},\underline{F}_1,\underline{w}_1,\dots,\underline{w}_L,\underline{F}_L,
\Gamma^{\rm s},\mathcal{M}^{\rm m}_E,r_1,\delta\right)$;
 \Comment{see Alg.~\ref{alg.born}}
\State $\mathcal{P}:=\{(\tilde{l},l)\in \{1,..,L\}^2\colon \tilde l<l\}$;
\For{$j=2,\dots, J$}
\State $\dscatamu{*}:= \texttt{scattering\_amplitude}(\underline{v}^*_E)$;
\Comment{Solve phased forward problem}
\State $\dscatamul{*}{l}:= \texttt{scattering\_amplitude}(\underline{v}^*_E
+\underline{w}_l)\quad \mbox{for }l=1..L$;
\State $\dscatapprox:=0$; \quad $\underline U^{**}_p:=0$ for $p \in \tilde \Gamma^m \cup \Gamma^e$;
\Comment{grid functions on $\widetilde{\Gamma}^{\rm m}\cup\Gamma^{\rm e}$}
\For{$p\in \widetilde{\Gamma}^{\rm m}$ such that $|p|<r_j$}
\State 
Choose $(\tilde{l}(p),l(p))\in\operatorname{argmax}_{(\tilde{l},l)\in\mathcal{P}} 
\Re \underline{\widehat w}_{\tilde{l},p} \Im \underline{\widehat w}_{l,p} - \Im \underline{\widehat w}_{\tilde{l},p} \Re \underline{\widehat w}_{l,p}$; 
\State $(\underline{\tilde f}^\text{appr}_{E})_p:=\texttt{Urec}\left(
(\Phi\underline{f}^*_{E,\tilde{l}(p)}-\Phi\underline f^*_E)_p,
(\Phi\underline f^*_{E,l(p)}-\Phi\underline f^*_E)_p,
(\Phi \underline{F}_0)_p,(\Phi\underline{F}_{\tilde{l}(p)})_p,
(\Phi \underline{F}_{l(p)})_p\right)$;
\State $\underline U^{**}_p:= (A \underline{v}^*_E)_p
+ (\underline{\tilde f}^\text{appr}_{E})_p-(\Phi\underline f^*_E)_p$;
\EndFor
\State $\underline{v}^*_E := (D^{1/2}\tilde{A})^{\dagger}D^{1/2}\underline U^{**}$;
\Comment{may be computed by CG and NFFT}
\EndFor 
\end{algorithmic}
\caption{function $\underline{v}^*_E$ = \texttt{phaseless\_iterative\_inv}($\underline F_0$, $\underline F_1$, $\underline w_1$,..,$\underline F_L$,
$\Gamma^{\rm s},\mathcal{M}^{\rm m}_E$,$\underline w_L$, $r_1..r_J$, $\delta$)}\label{alg.iter}
\end{algorithm}

\subsection{Numerical results}

\paragraph{Reconstruction errors.}

We consider the reconstruction of the potential $v$ shown in Fig.~\ref{fig.pot} (a)\footnote{this potential is given by the {\tt Matlab}'s function {\tt peaks}} using three background potentials $w_1$, $w_2$, $w_3$ shown at Fig.~\ref{fig.pot} (b), (c), (d).  The differential scattering cross-section of the 
potential $v$ for the incident direction $k=\sqrt E(-1,0)$ at $E=10^2$ and $E=20^2$ is shown at Fig.~\ref{fig.sca}. One can see that for bigger energy the differential scattering cross-section is more concentrated near $l \approx k$.

In the experiments $32$ equidistant incident directions and $256$ equidistant measurement directions 
(for each $k$) were used. Moreover we choose $\spcBd\approx c \sqrt E$ with $c=5$ so that the space grid discretization step is $2/\spcBd \approx 0.4 / \sqrt E$. For simplicity, we choose uniformly increasing cutoff radii $r_j = \sqrt{E}\bigl(1+ \frac{j}{J+1}\bigr)$.

\begin{figure}[ht]
	\begin{center}
	\begin{minipage}{.33\linewidth}
		   \begin{center}
		   \includegraphics[width=1\linewidth]{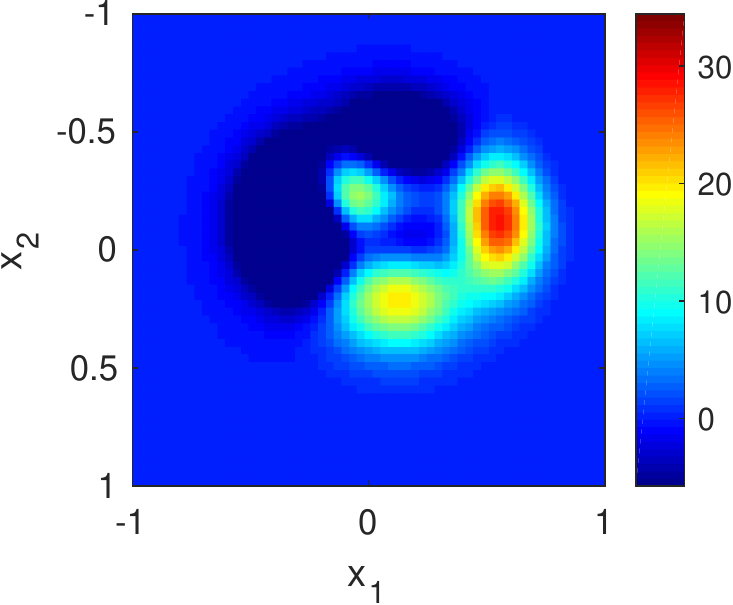} \\
		   (a) $v$ \\ \bigskip

		   \includegraphics[width=1\linewidth]{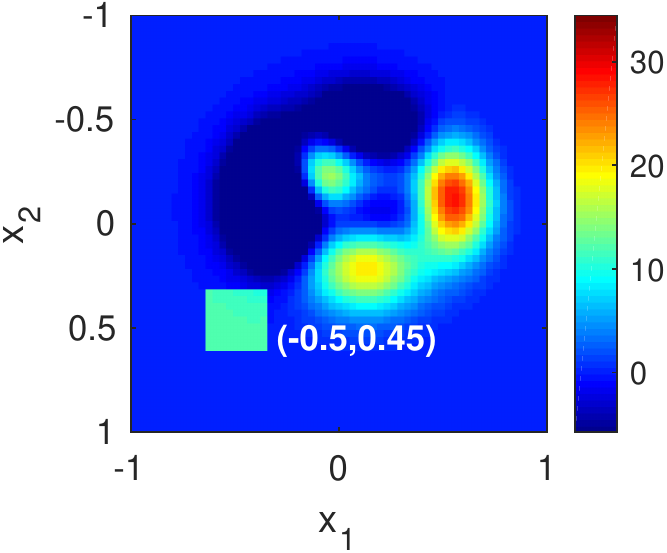} \\
		   (c) $v+w_2$
		   \end{center}
		   
	\end{minipage}\qquad
	\begin{minipage}{.33\linewidth}
			\begin{center}
			\includegraphics[width=1\linewidth]{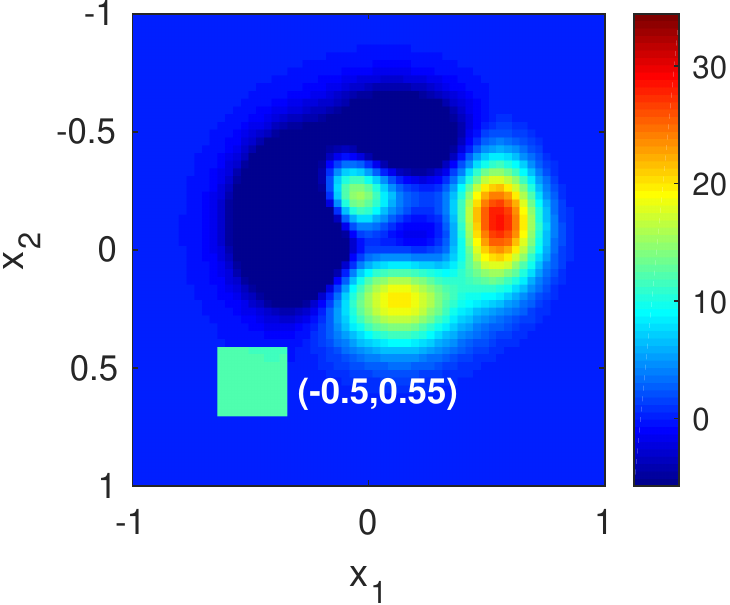}\\
			(b) $v+w_1$ \\\bigskip
			
			\includegraphics[width=1\linewidth]{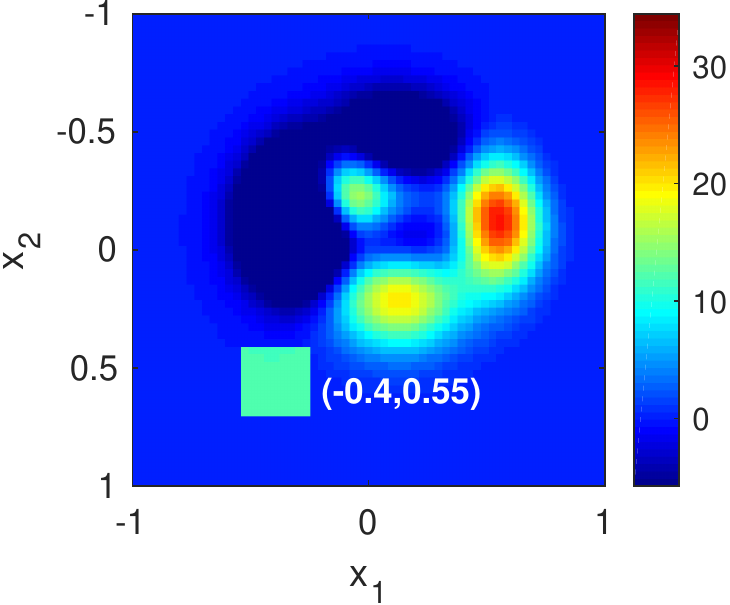}\\
			(d) $v+w_3$
			\end{center}
	\end{minipage}
	\end{center}
	\caption{Unknown potential and background potentials}\label{fig.pot}
\end{figure}

\begin{figure}[ht]
	\begin{center}
	\begin{minipage}{.4\linewidth}
		\begin{center}
		\includegraphics[width=.85\linewidth]{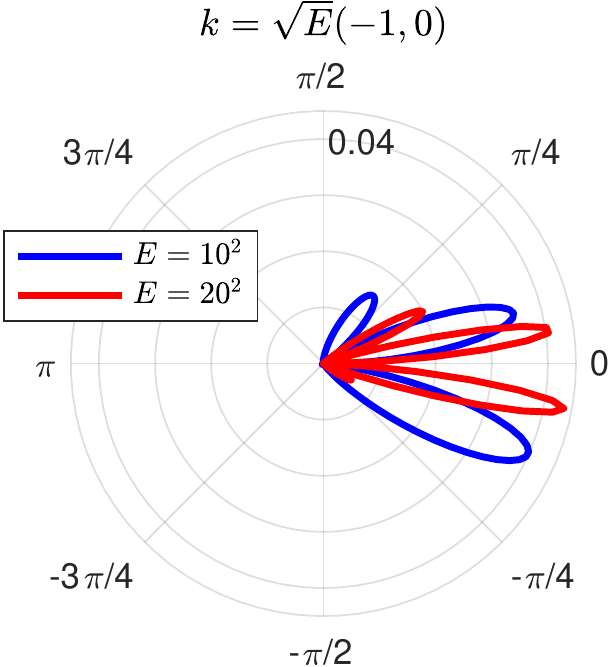}
		\end{center}
	\end{minipage}
	\begin{minipage}{.4\linewidth}
		\begin{center}
		\includegraphics[width=.83\linewidth]{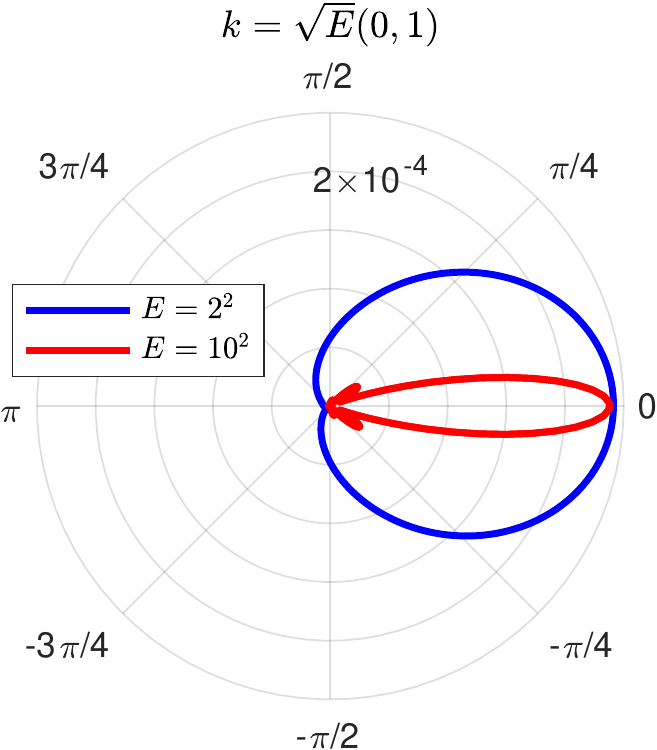}
		\end{center}
	\end{minipage}
	\end{center}
	\caption{Differential scattering cross-sections $|f(k,l)|^2$ for the potential of Fig.~\ref{fig.pot} (a) at $k=\sqrt E(-1,0)$ (left) and for the potential of Fig.~\ref{fig.nsm} (a) at $k=\sqrt E(0,1)$ (right). The distance to the curve in direction $l$ is equal to $|f(k,l)|^2$.}\label{fig.sca}
\end{figure}

In quantum mechanical and optical applications the dominant source of noise is often caused by 
the limited number $N_p$ of measured particles, leading to Poisson distributed data. More precisely, recall that the quantity $|f(k,l)|^2$ is proportional to the probability density of scattering of a particle with initial momentum $k$ into direction $l/|l| \neq k/|k|$. Let $K:=\#\{k_1,\dots,k_m\}$ denote the number
of incident waves. We assume that for each background potential and each incident wave the exposure 
time $t_l(k_m)$ is chosen such that the same expected number of particles $N_p/K(L+1)$ is recorded. 
Thus, our simulated noisy data 
$\underline{F}:=[\underline{F}_0, \underline{F}_1,\dots,\underline{F}_L]$, 
were generated from exact data 
$\underline{F}^\dagger:=[\underline{F}_0^\dagger, \underline{F}_1^\dagger,\dots,\underline{F}_L^\dagger]$
with 
$(\underline{F}^{\dagger}_0)_m:=|f(k_m,l_m)|^2$ and $(\underline{F}^{\dagger}_l)_m:=|f_l(k_m,l_m)|^2$, 
$l=1,\dots,L$ by
\begin{gather*}
  (\underline F_l)_m \sim \frac{1}{t_{l}(k_m)} \operatorname{Pois}\left( t_{l}(k_m) (\underline F_l^\dagger)_m \right) \quad \mbox{with }
  t_{l}(k_m) := \frac{N_p}{K(L+1)\sigma_l(k_m)} , \quad \sigma_l(k_m) := \sum_{n\colon k_n=k_m} (\underline F_l^\dagger)_n
\end{gather*}
(see \cite{HW:16} for more details). 
Here $\operatorname{Pois}(x)$ stands for a  Poisson random variable with mean $x$. 
Recall that 
$\mathbf{E}(\underline{F}_l)_m = (\underline{F}^{\dagger}_l)_m$ and that the pointwise noiselevel 
is $\sqrt{\mathbf{Var}(\underline{F}_l)_m}=t_{l}(k_m)^{-1/2}\sqrt{\vphantom{|}\smash{(\underline{F}_l^{\dagger})}_m}$. 
For the potential in Fig.~\ref{fig.pot} with $E=15^2$ the average pointwise noise level 
$\|\underline{F}-\underline{F}^{\dagger}\|_2/\|\underline{F}^{\dagger}\|_2$ was about 
44\%, 14\%, 4.4\% and 1.4\% for total count numbers $N_p\in\{10^7,10^8,10^9,10^{10}\}$.  
However, we stress that pointwise noise levels, although frequently used, 
are misleading as they tend to infinity as the discretization of the data space becomes 
finer and finer without loss of information, and that $N_p$ (or 
$N_p^{-1/2}$) is a better characterization of the noise level.

As the proposed method only yields an approximate solution at fixed energy even for noiseless data, 
a natural question is how much the reconstructions of our method can be improved by iterative 
regularization methods. Here we choose the Newton conjugate gradient method (NewtonCG) (see \cite{hanke:97b})
with $H^1$ inner product in the preimage space 
 as a commonly used representative of this class of methods, and use the 
results of our method as starting point of NewtonCG. The stopping index of the NewtonCG iteration 
was chosen by the discrepancy principle using the estimated noise level 
$\mathbf{E}\|\underline{F}-\underline{F}^{\dagger}\|_2^2
= \sum_{l,m} \mathbf{Var} (\underline{F}_l)_m= \sum_{l,m} t_l(k_m)^{-1}(\underline{F}^{\dagger}_l)_m 
\approx t_l(k_m)^{-1}(\underline{F}_l)_m$ (even though the discrepancy principle is actually only justified 
for deterministic noise models, see \cite{hanke:97b}).

The number of iterations $J$ is chosen basing on the following observations: for small particle count such as $N_p = 10^7$ after two or three iterations the accumulated noise in reconstruction by our method is already comparable with the reconstruction error and the method should be stopped; for bigger particle counts such as $N_p \geq 10^{10}$ the number of steps can be chosen in a relatively large range without significant impact on the reconstruction error; we use $J=8$ iterations for large particle counts 
$N_p\geq 10^8$.

Figure \ref{fig.css} shows cross-sections of the reconstructed potentials for the reconstructions using 
(a) the phaseless Born approximation, (b) our method and (c) our method combined with NewtonCG. In this experiment we used the values of parameters $E = 10^2$ and $N_p = 10^9$ and the highest possible 
scaling factor for the potential shown in Fig.~\ref{fig.pot} (a), for which our method still works at 
this energy level. This corresponds to $(L^1,L^2,L^\infty)$ norms $(13.5,14.3,37.0)$, respectively. 

Table \ref{tab.error} shows the reconstruction errors using the same methods for different energy 
levels and different expected count numbers $N_p$. Errors are averaged over $5$ experiments. 
These results demonstrate that our method performs well far beyond the scope of validity of the 
Born approximation. It can also be seen from these tables that the best reconstruction results are 
achieved by combining the proposed method with an iterative regularization method such as NewtonCG.

\begin{figure}[ht]
	\begin{minipage}{.32\linewidth}
		\begin{center}
			\includegraphics[width=.8\linewidth]{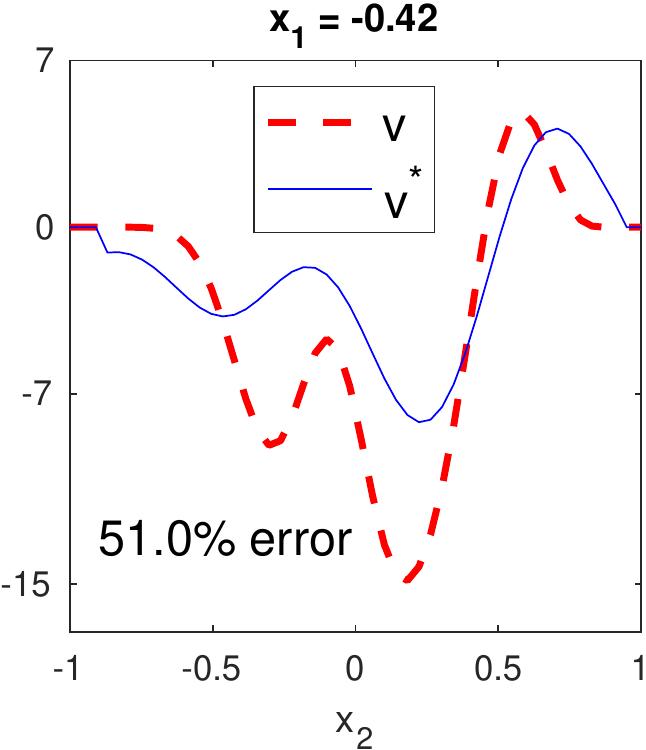} \\
			(a) Born approximation
		\end{center}
	\end{minipage}
	\begin{minipage}{.32\linewidth}
		\begin{center}
			\includegraphics[width=.8\linewidth]{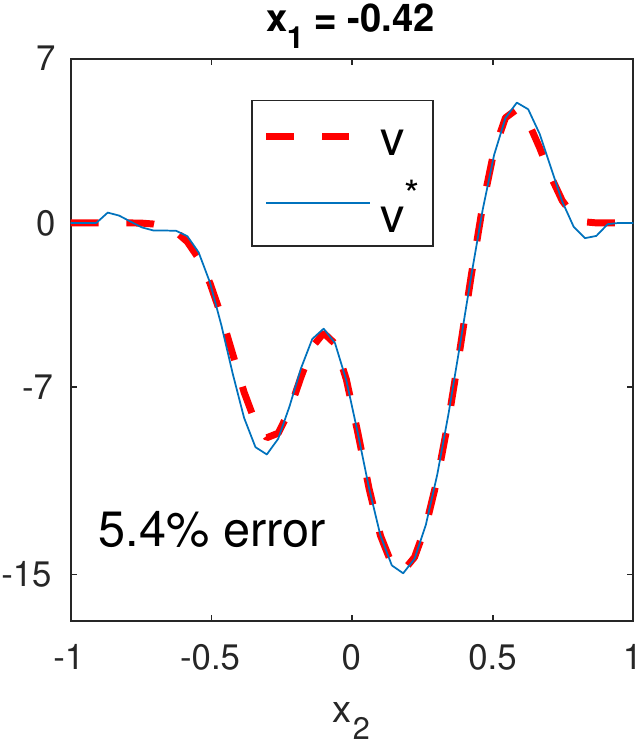} \\
			(b) Our method
		\end{center}
	\end{minipage}
	\begin{minipage}{.32\linewidth}
		\begin{center}
			\includegraphics[width=.8\linewidth]{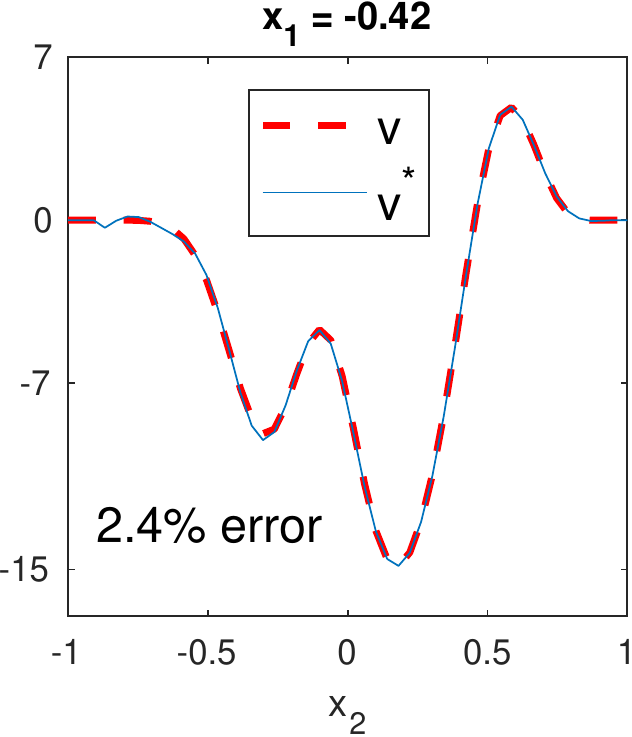} \\
			(c) Our method + NewtonCG
		\end{center}
	\end{minipage}
	\caption{Cross-sections of the exact $v$ and reconstructed potentials $v^*$, and $L^\infty$ relative reconstruction errors. Here $E = 10^2$ and $N_p=10^9$.}\label{fig.css}
\end{figure}

\begin{table}
	\begin{minipage}{.3\linewidth}
		\begin{center}
		\begin{tabular}{r|ccc}
		$N_p \backslash E$ & $10^2$ & $15^2$ & $20^2$ \\ \hline 
		$10^7$ & 53 & 20 & 16 \\ 
		$10^8$ & 53 & 19 & 16 \\ 
		$10^9$ & 53 & 19 & 16 \\ 
		$10^{10}$ & 53 & 20 & 16 \\ 
		\end{tabular} \\ \medskip
		(a) Born approximation
		\end{center}
	
	\end{minipage}
	\begin{minipage}{.3\linewidth}

    \begin{center}
	\begin{tabular}{r|ccc}
	$N_p \backslash E$ & $10^2$ & $15^2$ & $20^2$ \\ \hline 
	$10^7$ & 15 & 8.3 & 7.9 \\ 
	$10^8$ & 7.9 & 3.2 & 2.9 \\ 
	$10^9$ & 6 & 1.2 & 1.3 \\ 
	$10^{10}$ & 5.8 & 1.1 & 0.78 \\ 
	\end{tabular} \\ \medskip
	(b) Our method
	\end{center}
	\end{minipage}
    \begin{minipage}{.3\linewidth}
	\begin{center}	
	\begin{tabular}{r|ccc}
	$N_p \backslash E$ & $10^2$ & $15^2$ & $20^2$ \\ \hline 
	$10^7$ & 6 & 3.6 & 5.2 \\ 
	$10^8$ & 3.8 & 2 & 2.4 \\ 
	$10^9$ & 2.5 & 0.95 & 0.97 \\ 
	$10^{10}$ & 2 & 0.67 & 0.64 \\ 
	\end{tabular} \\ \medskip
	(c) Our method+NewtonCG
	\end{center}
	\end{minipage}
	
	\caption{relative $L^\infty$  reconstruction errors in percents for 
	the potential $v$ and the background potentials $w_1,w_2,w_3$ shown 
	in Fig.~\ref{fig.pot}.}\label{tab.error}
\end{table}

\paragraph{Comparison of convergence regions with NewtonCG.}
We also studied the influence of scaling of background potentials on reconstructions using our method and using NewtonCG. We made tests for the unknown potential of Fig.~\ref{fig.pot} (a) and three type-B potentials based on the Wendland functions with $k=1$. 

First, we tried to recover the potential using NewtonCG with zero initial guess. Simulations show that 
the iterations do not converge to the exact potentials unless the potential is downscaled by a factor of 
75 or bigger. On the other hand, our method works reasonably well, see Tab.~\ref{tab.bgp}, column 3B(W), 
and it can provide an initial guess from which NewtonCG converges. 

We also noticed that the simultaneous downscaling of the unknown and background potentials, even by a factor of 1000, does not solve the convergence problem of NewtonCG: the norm ratio of the background potential to the unknown potential must be also sufficiently small to guarantee the convergence of NewtonCG from the zero initial guess.

\paragraph{Non-smooth potentials.}
Consider the potential of \ref{fig.nsm} (a). The differential scattering cross-sections of this potential at energies $E = 2^2$ and $E=10^2$ for the incident direction $k=\sqrt E (0,1)$ is shown at Fig.~\ref{fig.sca} (right).

Fig.~\ref{fig.nsm} (b), (c) shows reconstructions of this non-smooth potential using the Born approximation and our method without NewtonCG. In the experiment we use background potentials $w_1$, $w_2$ of type A (i.e. $w_2 = i w_1$), the energy is $E=10^2$, and the particle count is $N_p = 10^{15}$. One can see that even though our method is theoretically justified only 
for sufficiently smooth potentials, it performs well also for non-smooth potentials.

\begin{figure}[ht]
	\begin{minipage}{.32\linewidth}
		\begin{center}
			\includegraphics[width=1\linewidth]{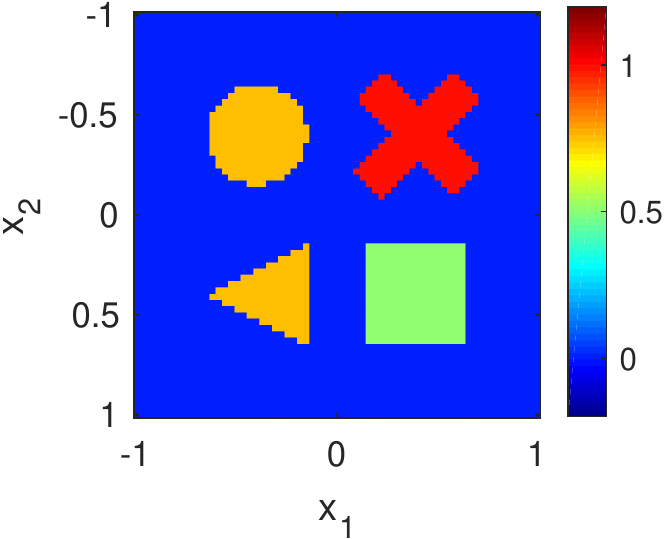} \\
			(a) Exact potential
		\end{center}
	\end{minipage}
	\begin{minipage}{.32\linewidth}
		\begin{center}
			\includegraphics[width=1\linewidth]{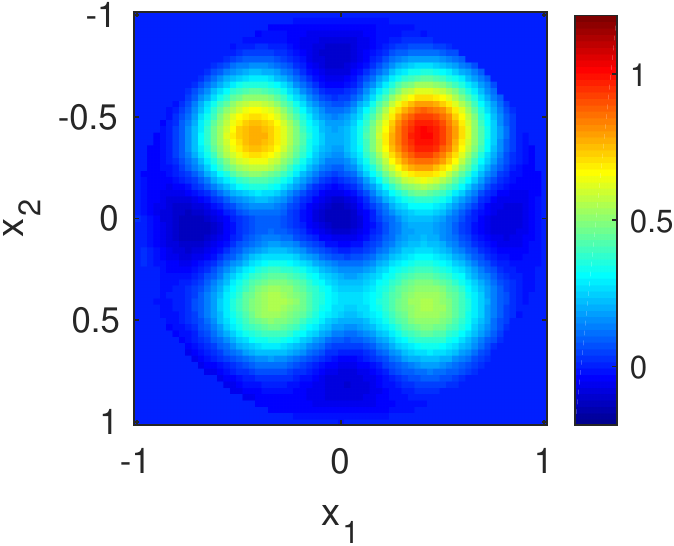} \\
			(b) Born approximation
		\end{center}
	\end{minipage}
	\begin{minipage}{.32\linewidth}
		\begin{center}
			\includegraphics[width=1\linewidth]{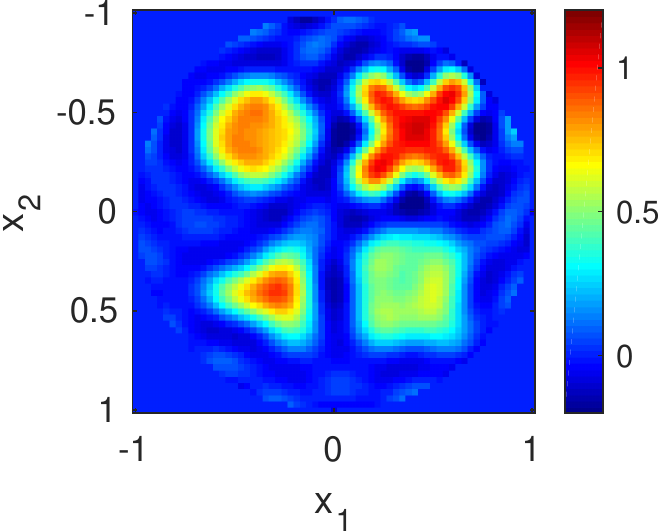} \\
			(c) Our method
		\end{center}
	\end{minipage}
	\caption{Reconstruction of a non-smooth potential for rectangular background potentials 
	of type A, $E=10^2$, and $N_p=10^{15}$.}\label{fig.nsm}
\end{figure}

\paragraph{Choice of background potentials.}
Examples of functions with compact support satisfying assumption \eqref{pex.defw} are 
Wendland's radial basis functions (see \cite{wendland:95}), in particular 
\begin{align}\label{eq:Wendland}
&w(x) = \Phi_k(|x|_2/h)\qquad \mbox{for some }h>0, k\in\{0,1\} \mbox{ and}\\
&\Phi_0(r):= \max(1-r,0)^2,\qquad \Phi_1(r):= \max(1-r,0)^4(4r+1).
\end{align}
More precisely, \eqref{pex.defw} is satisfied with $\sigma= d+2k+1$  
(see \cite{wendland:98}), and $w\in C^{2k}(\mathbb{R}^d)$.

\begin{table}[ht]
  \begin{center}
    \begin{tabular}{l|cccccc}
     & 2A(R) & 2A(W) & 3B(R) & 3B(W) & 2B(R) & 2B(W) \\ \hline 
    Born approximation& 20 & 20 & 19 & 21 & 20 & 25 \\ 
    our method& 2.2 & 3.4 & 1.2 & 5.8 & 6.5 & 16 \\ 
    our method+NewtonCG & 0.95 & 1.9 & 0.95 & 2.3 & 2 & 4.4 \\  \hline
    \end{tabular}
  \end{center}
  \caption{relative $L^\infty$ reconstruction errors in percents for different choices 
	of background potentials. 2A means two background potentials of type A, whereas 2B and 3B refers 
	to two or three background potentials of type B, respectively. (R) refers to indicator functions 
	of rectangles as in Fig.~\ref{fig.pot}, whereas (W) refers to Wendland functions \eqref{eq:Wendland} 
with $k=1$ with similar scaling. All simulations are performed for $E = 15^2$ and $N_p = 10^9$.}\label{tab.bgp}
\end{table}

	\begin{figure}[ht]
		\begin{center}
		\begin{minipage}{.28\linewidth}
			\begin{center}
				\includegraphics[width=0.95\linewidth]{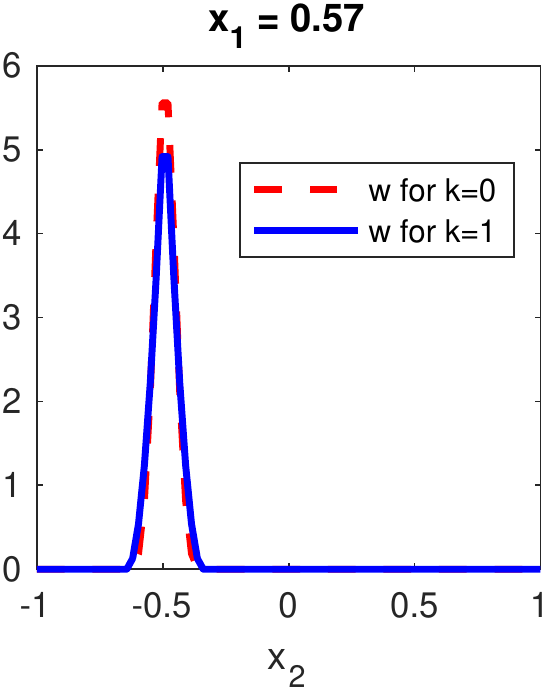}
			\end{center}
		\end{minipage}
		\begin{minipage}{.28\linewidth}
			\begin{center}
				\includegraphics[width=1\linewidth]{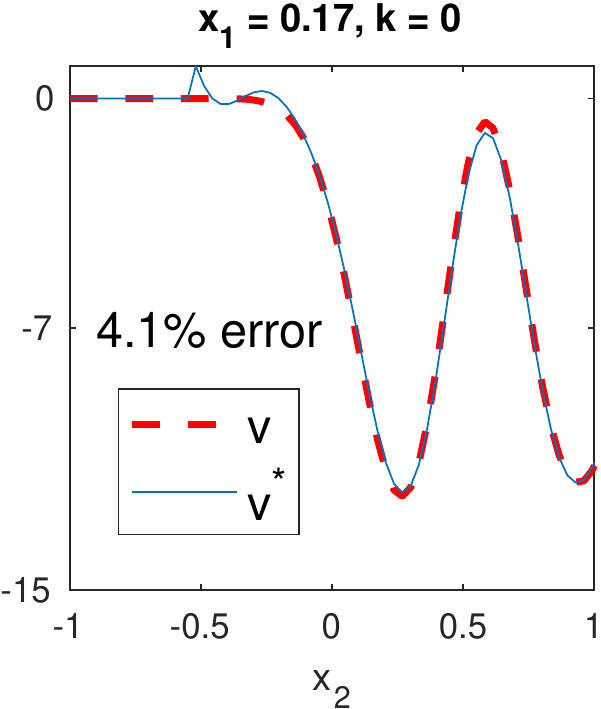}
			\end{center}
		\end{minipage}
		\begin{minipage}{.28\linewidth}
			\begin{center}
				\includegraphics[width=1\linewidth]{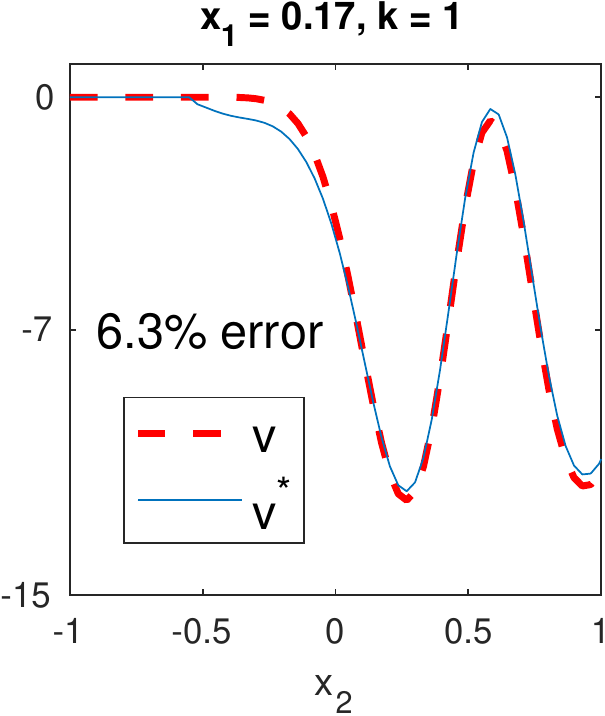}
			\end{center}
		\end{minipage}
		\end{center}
		\caption{Background potentials based on the Wendland functions $\Phi_0$, $\Phi_1$ defined in 
		\eqref{eq:Wendland} (left), cross-sections of the unknown potential $v$ and its reconstructions $v^*$ using our method for three type-B background potentials based on $\Phi_0$ (center) and for three type-B background potentials based on $\Phi_1$ (right)}\label{fig.wnd}
	\end{figure}
	
We also checked to which extent our approach still works if we use background potentials 
which do not satisfy assumption \eqref{pex.defw}, but may be easier to realize experimentally, 
such as indicator functions of squares. Our results are documented in Table \ref{tab.bgp} and 
Fig.~\ref{fig.wnd}. It can be seen that non-smooth background potentials yield better results 
than smooth background potential, but our method still works reasonably well for the 
$C^2$-Wendland function in \eqref{eq:Wendland}. Moreover, the best results are obtained 
for the nonsmooth rectangular potentials even though they do not satisfy assumption 
\eqref{pex.defw}. 

Furthermore, our method yields good results for two background potentials of type A, but 
significantly worse results for two background potentials of type B. However, the results 
for type B background potentials can be improved to a quality comparable to type A potentials 
if either a subsequent NewtonCG iteration is used or if data from a third shifted potential 
are available.

\paragraph{Robustness against errors in background potentials.}
In practice it is usually not possible to measure the background potentials exactly. Therefore, we 
tested our algorithm in the case where the simulated data are generated using a potential $\widetilde w$ 
which is a perturbation of the potential $w$ used in our reconstruction method. 
Figure \ref{fig.bgp} shows cross-sections of the potentials $w$ and $\widetilde w$, and a cross-section
of the reconstruction of $v$. In this example the potential $\widetilde w$ is obtained from $w$ by amplitude scaling ($\times 1.3$), support scaling ($\times 0.8$), translation by $(0.1,0)$, addition of Gaussian noise with standard deviation .22 and convolution with Gaussian kernel with standard deviation 0.5. We use the same unknown potential $v$ of Figure \ref{fig.pot} as before and set $E = 15^2$, $N_p = 10^9$.

In this example the phaseless Born approximation, our method and our method combined with 
NewtonCG give the relative $L^\infty$ errors $(64\%,3.9\%,1.5\%)$, respectively. 
This demonstrates a remarkable robustness of our method against errors in the reference potentials.

\begin{figure}[ht]
	\begin{minipage}{.3\linewidth}
		\begin{center}
			\includegraphics[width=.93\linewidth]{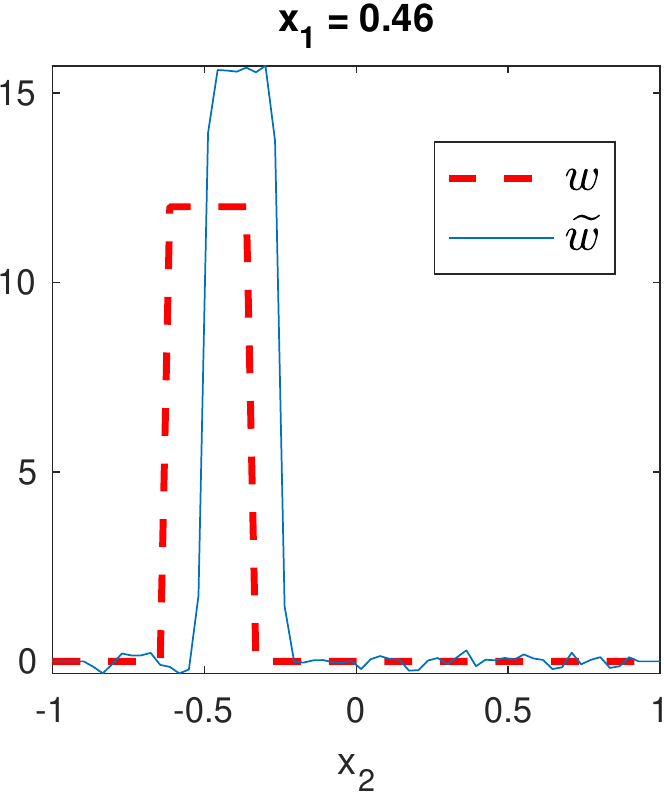} \\
		\end{center}
	\end{minipage}
	\begin{minipage}{.3\linewidth}
		\begin{center}
			\includegraphics[width=1\linewidth]{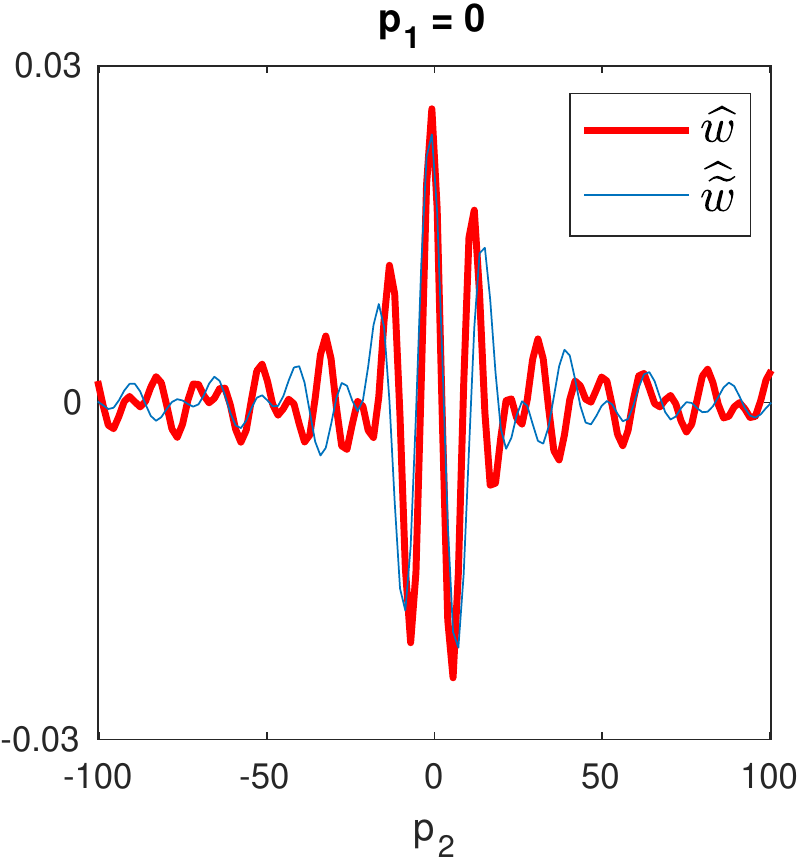} \\
		\end{center}
	\end{minipage}
	\begin{minipage}{.3\linewidth}
		\begin{center}
			\includegraphics[width=.93\linewidth]{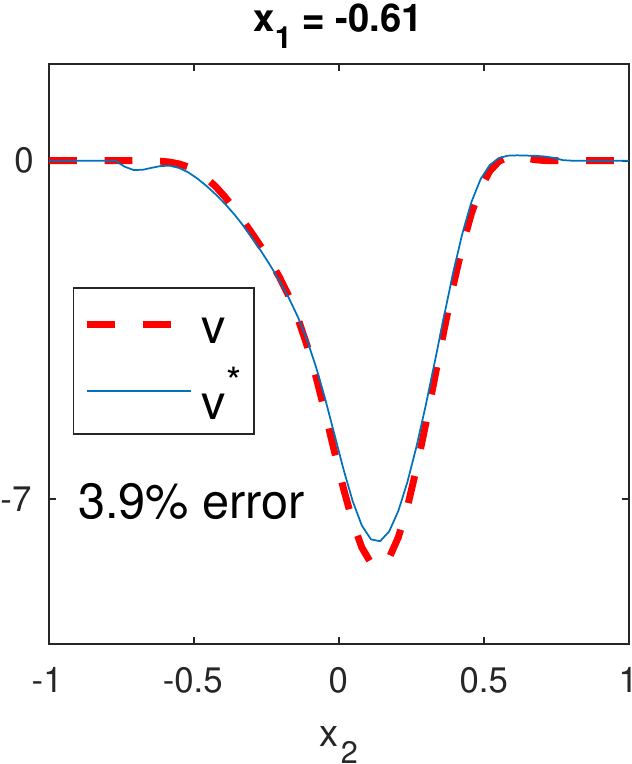} \\
		\end{center}
	\end{minipage}
	\caption{Left and center: Cross-sections of the background potential $w$ used in the reconstruction 
	method
	and the background potential $\widetilde w$ used to generate simulated data: (left) space domain; (center) Fourier domain. Right: Cross-sections of the unknown potential $v$ and the potential $v^*$ recovered by our method.}\label{fig.bgp}
\end{figure}

\section{Conclusions}
We have proposed a method for the solution of phaseless inverse medium scattering problems in the 
presence of known background potentials. Let us summarize the  advantages and disadvantages of our method 
in comparision with iterative regularzation methods such as NewtonCG:
\begin{itemize}
	\item \textit{Global convergence:} Iterative regularization methods require a good initial approximation to the unknown potential, whereas we have shown global convergence of our method as the 
	energy tends to infinity. Numerical experiments at fixed energy demonstrate excellent performance 
	of our method for large potentials and/or weak background potentials where NewtonCG failed. 
		\item \textit{Computation time.} Each iteration step of our reconstruction method requires only 
		($L+1$) solutions of a forward problem, where $L$ is the number of background potentials, and other comparatively cheap operations. In contrast, regularized 
		Newton methods additionally require the solution of a linearized inverse problem in each 
		iteration step, and they typically need a larger number of iteration to achieve the accuracy of 
		our method. In our experiments the proposed method was typically more than 20 times faster than 
		NewtonCG, but we stress that the quotient of computation times strongly depends on the noise level, 
		the energy, the potentials, the choice of the direct solver (setup time vs.~solution time), 
		and other parameters. 
	\item \textit{Use of black-box solvers.} As our method only requires the solution of forward 
	scattering problems, any block-box solver for such problems can be used. In contrast, 
	iterative regularization methods additionally use the Fr\'echet derivative 
	of the forward operator and typically also the adjoint of the Fr\'echet derivative. 
	The implementation of these operations may require modifications of the source code of the forward solver. 
	\item \textit{Asymptotic exactness.} At fixed energy in the absence of noise the NewtonCG is expected to converge to the exact solution under some additional assumptions, in particular the uniqueness of solution and the  tangential cone condition, see \cite{hanke:97b}. In contrast, theoretically our method will converge to the exact potential only in the limit $E \to + \infty$.
	\item \textit{Stopping rules.} There exists a considerable literature on a-posteriori stopping rules 
	for iterative regularization methods and their convergence properties such as the discrepancy 
	principle for NewtonCG (see \cite{hanke:97b}). In contrast, we have used a rather ad-hoc a-priori 
	stopping rule in our experiments with the proposed method for the lack of a better alternative. 
\end{itemize}
This discussion shows that the pros and cons of our method are rather complementary to those of iterative 
regularization methods. Therefore, the proposed method provides a valuable new tool for the solution of 
phaseless inverse medium scattering problems. Hybrid methods using the proposed method to compute 
an initial guess for an iterative regularization method allow to combine the advantages of both methods. 
But in many cases the proposed method itself may already provide a sufficiently accurate reconstruction. 

On the theoretical side we have demonstrated fast convergence of our method as energy tends to infinity 
for exact data and two types of background potentials. It remains for future research to study the 
behavior of the proposed method in the presence of noise and to devise and analyze useful stopping rules.

\appendix
\section{Proofs}
\subsection{Proof of Lemma \ref{isp.lemmadisp}}\label{sec.dis}
Under the assumptions of Lemma \ref{isp.lemmadisp}, the following estimates hold:
\begin{equation}\label{dis.modest}
  \begin{gathered}
      \left| |\Phi\scatam(p) + \Phi\scatamul{*}{j}(p)-\Phi\scatam(p) | - 			|\Phi \scataml{j}(p)| \right| \leq \muc{sec.dis}{1} E^{-\alpha-\frac 1 2}, \\
      p \in \mathcal B^d_{2\sqrt E}, \; E \geq E^*, \; j=1,2,
  \end{gathered}
\end{equation}
where $\muc{sec.dis}{1} > 0$. The proof of estimates \eqref{dis.modest} is based on the following inequalities:
\begin{equation}\label{dis.modest2}
 \begin{aligned}
    \bigl| |\Phi\scatam + \Phi\scatamul{*}{l} - \Phi\scatamu{*}| 
		- |\Phi\scataml{j}| \bigr| 
		& \leq | \Phi\scatam + \Phi\scatamul{*}{l} - \Phi\scatamu{*} 
		- \Phi\scataml{l} | \\
    & \hspace{-6em} = \bigl| (\Phi\scatam - \widehat v) - (\Phi\scatamu{*} - \widehat{v^*_E}) 
		- \bigl( (\Phi\scataml{l} - \widehat{v_{E,l}}) 
		- (\Phi\scatamul{*}{l} - \widehat{v_{E,l}^*}) \bigr) \bigr| \\
    & \hspace{-6em} \leq \bigl| (\Phi\scatam-\widehat v) 
		- (\Phi\scatamu{*} - \widehat{v^*_E}) \bigr| 
		+ \bigl| (\Phi\scataml{l} - \widehat{v_l}) 
		- (\Phi\scatamul{*}{l} - \widehat{v_{E,l}^*}) \bigr|
 \end{aligned}
\end{equation}
for $l = 1,2$.
Applying Lemma \ref{itp.lemma} to $v$, $v^*_E$ and to $v_l$, $v_{E,l}^*$, respectively, we get the estimates
\begin{equation}\label{dis.modest3}
  \begin{aligned}
    \bigl| (\Phi\scatam(p)-\widehat v(p)) - (\Phi\scatamu{*}(p) - \widehat v^*_E(p)) \bigr| & \leq \tfrac 1 2 \muc{sec.dis}{1} E^{-\alpha-\frac 1 2},\\
    \bigl| (\Phi\scataml{l}(p)-\widehat{v_l}(p)) - (\Phi\scatamul{*}{l}(p)
		-\widehat v_{E,l}^*(p)) \bigr| 
		& \leq \tfrac 1 2 \muc{sec.dis}{1} E^{-\alpha-\frac 1 2}, \quad l = 1,2,
  \end{aligned}
\end{equation}
where $\tfrac 1 2 \muc{sec.dis}{1} > 0$ is the constant in the right-hand side of formula (3.14) of \cite{Nov2015en}, $p \in \mathcal B^d_{2\sqrt E}$, $E \geq E^*$. Formula \eqref{dis.modest} follows from \eqref{dis.modest2}, \eqref{dis.modest3}.

Next, using \eqref{dis.modest} we get the following estimate:
\begin{equation}\label{dis.modest4}
  \begin{gathered}
  \bigl| |\Phi\scatam(p) + \Phi\scatamul{*}{l}(p)-\Phi\scatamu{*}(p)|^2 
	- |\Phi\scataml{l}(p)|^2 \bigr| \leq \muc{sec.dis}{1} \muc{sec.dis}{2} E^{-\alpha-\frac 1 2},\\
  p \in \mathcal B^d_{2\sqrt E}, \; E \geq E^*, \; l = 1,2,
  \end{gathered}
\end{equation}
for $\muc{sec.dis}{2}$ such that
\begin{equation}\label{dis.ampsbnds}
    \max\bigl\{ |\Phi\scatam|, |\Phi\scataml{1}|, |\Phi\scataml{2}|, 
		|\Phi\scatamul{*}{1}-\Phi\scatamu{*}|, |\Phi\scatamul{*}{2}-\Phi\scatamu{*}| \bigr\} 
		\leq \tfrac 1 3 \muc{sec.dis}{2},
\end{equation}
where $p \in \mathcal B^d_{2\sqrt E}$, $E \geq E^*$. The estimate \eqref{dis.modest4} can be rewritten as
\begin{equation}\label{dis.modest5}
  \begin{gathered}
    \bigl| 2\Re (\Phi\scatamul{*}{l} - \Phi\scatamu{*}) \Re \Phi\scatam(p) 
		+ 2 \Im (\Phi\scatamul{*}{l} - \Phi\scatamu{*}) \Im \Phi\scatam(p) \\ 
    - \bigl(|\Phi\scataml{l}|^2-|\Phi\scatam|^2 - |\Phi\scatamul{*}{l}-\Phi\scatamu{*}|^2 \bigr) \bigr| \leq \muc{sec.dis}{1} \muc{sec.dis}{2} E^{-\alpha-\frac 1 2}, \quad l = 1,2.
  \end{gathered}
\end{equation}
In turn, one can rewrite \eqref{dis.modest5} in matrix form as
\begin{gather}
   \biggl| 2 M(p,E) \cdot 
   \begin{pmatrix}
      \Re \Phi\scatam(p) \\
      \Im \Phi\scatam(p) 
    \end{pmatrix} - b(p,E) \biggr| \leq \muc{sec.dis}{1} \muc{sec.dis}{2} E^{-\alpha-\frac 1 2} \label{dis.disest}
\end{gather}
with the matrix $M$ and the vector $b$ from Algorithm \ref{alg:Urec}. 
Using \eqref{isp.zetaest}, \eqref{isp.zetageqeps} and \eqref{isp.E**def}, we obtain the estimate \eqref{isp.zetaEgeqeps}.

Using the formula for the inverse matrix, we also get the following equality:
\begin{equation}\label{dis.Minvest}
  \bigl\| M^{-1}(p,E) \bigr\|_F 
	= \frac{\bigl( |\Phi\scatamul{*}{1}-\Phi\scatamu{*}|^2+|\Phi\scatamul{*}{2}-\Phi\scatamu{*}|^2 \bigr)^{\frac 1 2}}{|\zeta^*(p,E)|}, \quad  E \geq E^{**},
\end{equation}
for $p \in \mathcal B^d_{2\sqrt E}$ such that \eqref{isp.zetageqeps} holds and where $\|\cdot\|_F$ denotes the Frobenius matrix norm:
\begin{equation*}
  \|A\|_F = \sqrt{\sum_{i,j=1}^n |a_{ij}|^2} \quad \text{for any square matrix $A = (a_{ij})_{i,j=1}^n$}.
\end{equation*}

Using  \eqref{isp.zetaEgeqeps}, \eqref{dis.ampsbnds}, \eqref{dis.Minvest} we obtain the estimate
\begin{equation}\label{dis.Minvest2}
  \bigl\| M^{-1}(p,E) \bigr\|_F \leq \tfrac{2 \sqrt 2}{3} \delta^{-1} \muc{sec.dis}{2}, \quad E \geq E^{**}.
\end{equation}
Finally, using \eqref{dis.disest} and \eqref{dis.Minvest2}, we get the estimate \eqref{isp.dispest}, where
\begin{equation}\label{dis.mu4def}
  \muc{sec.isp}{3} = \tfrac 2 3 \muc{sec.dis}{1} \pmuc{sec.dis}{2}^2. 
\end{equation}
Lemma \ref{isp.lemmadisp} is proved.

\subsection{Proof of Theorem \ref{itc.thmc}}\label{sec.tcp}

\begin{proposition}\label{tcp.propest} Let $v$ and $w_1$, $w_2$ be the same as in \eqref{in.vprop}, \eqref{pex.defw}, \eqref{pex.defw1w2r}. Let $v^*_E$ be an approximation to $v$ satisfying \eqref{eqs:itp.v*}. 
Let $U^{**}_E(p)$ be defined according to \eqref{itc.U**def}, 
Then
\begin{equation}\label{tcp.U**est}
  \begin{gathered}
  |U^{**}_E(p)-\widehat v(p)| \leq \muc{sec.tcp}{1} E^{-\alpha - \frac 1 2} r^{2\sigma}, \quad \muc{sec.tcp}{1} :=  4^\sigma \pmuc{sec.pex}{1}^{-2} \muc{sec.isp}{3} + \muc{sec.itp}{1}, \\
  E \geq \max \{ \muc{sec.tcp}{2} r^{4\sigma}, E^*\}, \quad \muc{sec.tcp}{2} := 4^{2\sigma+1} \pmuc{sec.pex}{1}^{-4} \pmuc{sec.isp}{1}^2,\\
  p \in \mathcal B^d_r, \quad 1 \leq r \leq 2 \sqrt E,
  \end{gathered}
\end{equation}
where $\sigma$ is the same as in \eqref{pex.defw}, $E^* = E^*(K,D_\text{ext})$ is defined according to \eqref{itp.E*def}, \eqref{isp.Dextdef}, and $\mu_{k,\S X}$, $k \geq 1$, are the constants of Section X.  
\end{proposition}
\begin{proof}[Proof of Proposition \ref{tcp.propest}]
Due to \eqref{itp.E*def}, \eqref{itc.U**def} and Lemma \ref{isp.lemmadisp}, we have
\begin{equation}\label{tcp.U**est2}
\begin{gathered}
  |U^{**}_E(p)-\widehat v(p)| \leq (\muc{sec.isp}{3} \delta^{-1} + \muc{sec.itp}{1}) E^{-\alpha-\frac 1 2}, \\
  \text{if} \; p \in \mathcal B^d_{2\sqrt E}, \; |\zeta_{\widehat{w_1},\widehat{w_2}}(p)| \geq \delta, \; E \geq E^{**}, \\
  E^{**} = \max\bigl( 4 \pmuc{sec.isp}{1}^2 \delta^{-2} , E^* \bigr).
\end{gathered}
\end{equation}
Besides, in view of \eqref{pex.zetar}, we also have that
\begin{equation}\label{tcp.zetaest}
  \min_{p \in \mathcal B^d_r} |\zeta_{\widehat{w_1},\widehat{w_2}}(p)| = \pmuc{sec.pex}{1}^2 (1+r)^{-2\sigma} \geq \pmuc{sec.pex}{1}^2 4^{-\sigma} r^{-2\sigma}, \quad 1 \leq r.
\end{equation}
Using \eqref{tcp.U**est2} with $\delta = \pmuc{sec.pex}{1}^2 4^{-\sigma} r^{-2\sigma}$ and \eqref{tcp.zetaest}, we get \eqref{tcp.U**est}.
\end{proof}

We represent $v$ as follows:
\begin{equation}\label{tcp.vdec}
    \begin{gathered}
        v(x) = v^+(x,r)+v^-(x,r), \quad x \in \mathbb R^d, \; r > 0,\\
        v^+(x,r) = \int_{\mathcal B^d_r} e^{-ipx} \widehat v(p) \, dp, \\
        v^-(x,r) = \int_{\mathbb R^d \setminus \mathcal B^d_r} e^{-ipx} \widehat v(p) \, dp.
    \end{gathered}   
\end{equation}
Since $v \in W^{n,1}(\mathbb R^d)$, $n > d$, we have 
\begin{equation}\label{tcp.v-est}
  \begin{gathered}
        |v^-(x,r)| \leq \muc{sec.tcp}{3} \|v\|_{n,1} r^{d-n}, \quad x \in \mathbb R^d, \; r > 0, \\
        \muc{sec.tcp}{3} := |\mathbb S^{d-1}|\tfrac{(2\pi)^{-d} d^n}{n-d},
  \end{gathered}
\end{equation}
where $\|\cdot\|_{n,1}$ is defined in \eqref{in.Wn1def}, and $|\mathbb S^{d-1}|$ is the standard Euclidean volume of $\mathbb S^{d-1}$; see \cite{Agal2016b}.

Using \eqref{tcp.U**est} we obtain
\begin{equation}\label{tcp.v+est}
\begin{gathered}
        \left| v^+(x,r)-\int_{\mathcal B^d_r} e^{-ipx} U_{\widehat{w_1},\widehat{w_2}}(p,E) \, dp \right| \leq \muc{sec.tcp}{1} |\mathcal B^d_1| E^{-\alpha - \frac 1 2} r^{d+2\sigma} \\
        \text{for $x \in D$, $1 \leq r \leq 2\sqrt E$, $E \geq \max\{\muc{sec.tcp}{2} r^{4\sigma}, E^*\}$},
\end{gathered}
\end{equation}
where $|\mathcal B^d_1|$ denotes the standard Euclidean volume of $\mathcal B^d_1$.

It follows from \eqref{tcp.vdec}, \eqref{tcp.v-est}, and \eqref{tcp.v+est} that
\begin{equation}\label{tcp.eru1}
  \begin{gathered}
         \biggl|v(x) - \int_{\mathcal B^d_r} e^{-ipx} U_{\widehat{w_1},\widehat{w_2}}(p,E) \, dp\biggr| 
         \leq \muc{sec.tcp}{3} \|v\|_{n,1} r^{d-n} + \muc{sec.tcp}{1} |\mathcal B^d_1| E^{-\alpha-\frac 1 2} r^{d+2\sigma} \\
         \text{for $x \in D$, $1 \leq r \leq 2\sqrt E$, $E \geq \max\{\muc{sec.tcp}{2} r^{4\sigma}, E^*\}$}.
   \end{gathered}
\end{equation}
In addition, if $r = r(E)$, where $r(E)$ is the radius of \eqref{itc.beta1def}, then
\begin{gather}
  \begin{aligned}\label{tcp.r1def}
	r^{d-n} & = (2\tau)^{d-n} E^{-\beta}, \\
	E^{-\alpha - \frac 1 2} r^{d+2\sigma} & = (2\tau)^{d+2\sigma} E^{-\beta},
  \end{aligned}\\
	\begin{gathered}\label{tcp.EgeqrE}
  	 E \geq \muc{sec.tcp}{2} r(E)^{4\sigma} \quad \text{if $E \geq \muc{sec.tcp}{4}$},\\
	\muc{sec.tcp}{4} := \bigl( \muc{sec.tcp}{2} (2\tau)^{4\sigma}\bigr)^{\frac{n-d}{n-d-4\beta \sigma}}.
  	\end{gathered}
\end{gather}
Using formulas \eqref{itc.beta1def}, \eqref{tcp.eru1}, \eqref{tcp.r1def}, \eqref{tcp.EgeqrE}, we obtain
\begin{equation}\label{tcp.v**est}
 \begin{aligned}
	&|v^{**}_E(x)-v(x)| \leq B_1 E^{-\beta} \quad \mbox{for }x \in D, \; E \geq E_1 \quad \mbox{with}\\
	&B_1 := (2\tau)^{d-n} \muc{sec.tcp}{3} \|v\|_{n,1} + (2\tau)^{d+2\sigma} \muc{sec.tcp}{1} |\mathcal B^d_1|, \\
	&E_1 := \max\{ \muc{sec.tcp}{4}, E^*\}.
 \end{aligned}
\end{equation}

Theorem \ref{itc.thmc} is proved.

\begin{remark} We have that
\begin{equation}\label{tcp.tau1def}
  \begin{gathered}
  \text{$E_1 = E^*$ for $\tau \leq \tau_1$} \qquad \mbox{for }
  \tau_1 := \tfrac 1 2 \pmuc{sec.tcp}{2}^{-\frac{1}{4\sigma}} (E^*)^{1 - \frac{4\beta \sigma}{n-d}}.
  \end{gathered}
\end{equation}
\end{remark}

\subsection{Proof of Theorem \ref{itr.thmr}}\label{sec.trp}
\begin{proposition}\label{trp.propinterp} Let $v$ satisfy \eqref{in.vprop} and $v \in W^{n,1}(\mathbb R^d)$. Let $y \in \mathbb R^d$, $y \neq 0$, and $N \geq 1$. Put
\begin{equation}\label{trp.defVe}
 \begin{gathered}
	V^{N,\varepsilon}(p^\varepsilon(p_\bot,z,t)) = \sum_{-N \leq j \leq N, j \neq 0} \widehat v \bigl(p^\varepsilon(p_\bot,z,j)\bigr) L_j(t), \\
	p_\bot \in \mathbb R^d, \; p_\bot \cdot y = 0, \; z \in \mathbb Z, \; t \in (-1,1), \; 0 < \varepsilon < \min\{\tfrac 1 2 |y|, \tfrac{\pi}{N+1}\},
 \end{gathered}
\end{equation}
where $p^\varepsilon(p_\bot,z,t)$, $L_j$ are defined in \eqref{itr.defpzt}, \eqref{itr.deflj}. Then
\begin{equation}\label{trp.Vest1}
 \begin{gathered}
	\bigl| V^{N,\varepsilon}(p) - \widehat v(p) \bigr| \leq \muc{sec.trp}{1} \varepsilon^{2N} \bigl( 1 + |p_\bot| + \tfrac{\pi}{|y|} |z| \bigr)^{-n}, \\
	p = p^\varepsilon(p_\bot,z,t), \; p_\bot \in \mathbb R^d, \; p_\bot \cdot y = 0, \; z \in \mathbb Z, \; t \in (-1,1), \\
	\muc{sec.trp}{1} := \tfrac{2^{n-d}}{\pi^d} (1+d)^n \tfrac{(2N)^{2N}}{(2N)!|y|^{4N}} \|(x\cdot y)^{2N} v(x)\|_{n,1}.
 \end{gathered}
\end{equation} 
\end{proposition}

\begin{proof}[Proof of Proposition \ref{trp.propinterp}.] Let $p_\bot \in \mathbb R^d$, $p_\bot \cdot y = 0$ and $z \in \mathbb Z$ be fixed. The following estimate holds:
\begin{equation}\label{trp.Vest2}
 \begin{gathered}
  |V^{N,\varepsilon}(p)-\widehat v(p)| \leq \tfrac{(2N)^{2N}}{(2N)! |y|^{4N} } \varepsilon^{2N} \hspace{-1em} \sup_{s \in [-N,N]} \bigl| \mathcal ( [y \nabla]^{2N} \widehat v)(p^\varepsilon(p_\bot,z,s))\bigr|, \\ p = p^\varepsilon(p_\bot,z,t), \; t \in (-1,1), \;
  (y\nabla \varphi)(\xi) = y_1 \tfrac{\partial \varphi(\xi)}{\partial \xi_1} + \cdots + y_d \tfrac{\partial \varphi(\xi)}{\partial \xi_d}.
 \end{gathered}
\end{equation}
Estimate \eqref{trp.Vest2} follows from the formula
\begin{equation}
  \tfrac{\partial}{\partial s} p^\varepsilon(p_\bot,z,s) = \varepsilon \tfrac{y}{|y|^2},
\end{equation}
from the fact that $P(t) = V^{N,\varepsilon}(p^\varepsilon(p_\bot,z,t))$ in the Lagrange interpolating polynomial for $u(t) = \widehat v(p^\varepsilon(p_\bot,z,t))$ with the nodes at $t = \pm 1$, \dots, $\pm N$, and from the following standard estimate for the Lagrange interpolating polynomial $P(t) = V^{N,\varepsilon}(p^\varepsilon(p_\bot,z,t))$:
\begin{equation}
  \begin{gathered}
  |V^{N,\varepsilon}(p)-\widehat v(p)| \leq \tfrac{(2N)^{2N}}{(2N)!} \hspace{-1em} \sup_{s \in [-N,N]} \bigl| \tfrac{\partial^{2N}}{\partial s^{2N}}\widehat v(p^\varepsilon(p_\bot,z,s))\bigr|, \\
  p=p^\varepsilon(p_\bot,z,t), \quad t \in [-N,N].
  \end{gathered}
\end{equation}

In addition, the following estimate was proved in \cite{Agal2016b}:
\begin{equation}
  (1+|p|)^n |\widehat v(p)| \leq (2\pi)^{-d} (1+d)^n \|v(x)\|_{n,1}, \quad p \in \mathbb R^d.
\end{equation}
If we replace $v(x)$ by $( x \cdot y )^{2N} v(x)$, we get
\begin{equation}\label{trp.Dvest}
  \begin{gathered}
	(1+|p|)^n \bigl| [y\nabla]^{2N} \widehat v(p) \bigr| \leq (2\pi)^{-d} (1+d)^n \| (x \cdot y)^{2N} v(x)\|_{n,1}, \quad p \in \mathbb R^d.
   \end{gathered}
\end{equation}
Besides, we also have that
\begin{equation}\label{trp.pZeest}
  1+ |p| \geq \tfrac 1 2 \bigl( 1 + |p_\bot| + \tfrac{\pi}{|y|} |z(p)| \bigr), \quad p \in Z^\varepsilon_y,
\end{equation}
where we have used that $\varepsilon < \tfrac 1 2 |y|$. Using \eqref{trp.Vest2}, \eqref{trp.Dvest}, \eqref{trp.pZeest}, we get \eqref{trp.Vest1}.

Proposition \ref{trp.propinterp} is proved.
\end{proof}

\begin{proposition}\label{trp.propest} 
Let $v$ and $w_1$, $w_2$ be the same as in \eqref{in.vprop}, \eqref{pex.defw}, \eqref{pex.defw1w2r}. Let $v^*_E$ be an approximation to $v$ satisfying \eqref{eqs:itp.v*}. 
Let $U^{**}_E$ be defined according to \eqref{itc.U**def}, 
and let $U^*_{N,\varepsilon}(p,E)$ be defined according to \eqref{itr.defpzt}, \eqref{itr.defU**e}, \eqref{itr.U**eass}, \eqref{itr.deflj}. Then
\begin{gather}\label{trp.U**est}
  \begin{gathered}
  \bigl| \widehat v(p) - U^{**}_E(p) \bigr| \leq \muc{sec.trp}{2} E^{-\alpha-\frac 1 2}r^{2\sigma} \varepsilon^{-1}, \quad \muc{sec.trp}{2} :=  4^\sigma \tfrac{\pi}{2} \pmuc{sec.pex}{1}^{-2} \muc{sec.isp}{3} + \muc{sec.itp}{1},
  \end{gathered}\\
  \begin{gathered}\label{trp.U**estas}
  \text{for} \;\; p \in \mathcal B^d_r \setminus Z^\varepsilon_y, \;\;\;
  0 < \varepsilon < 1, \quad 1 \leq r \leq 2 \sqrt E, \\
  E \geq \max \bigl( \muc{sec.trp}{3} r^{4\sigma} \varepsilon^{-2} , E^* \bigr), \quad \muc{sec.trp}{3} := 4^{2\sigma}\pi^2 \pmuc{sec.pex}{1}^{-4} \pmuc{sec.isp}{1}^2, 
  \end{gathered}
\end{gather}
where $\sigma$ is the same as in \eqref{pex.defw}, $E^* = E^*(K,D_\text{ext})$ is defined according to \eqref{itp.E*def}, \eqref{isp.Dextdef}, and $\mu_{k,\S X}$, $k \geq 1$, are the constants of Section X.

In addition, if $v \in W^{n,1}(\mathbb R^d)$, $n \geq 0$, then
\begin{gather}\label{trp.U**est2}
 \begin{gathered}
	\bigl| \widehat v(p) - U^{**}_{N,\varepsilon}(p,E) \bigr|  
	 \leq \muc{sec.trp}{4} E^{-\alpha-\frac 1 2}r^{2\sigma} \varepsilon^{-1} + \muc{sec.trp}{1} \varepsilon^{2N} \bigl( 1 + |p_\bot| + \tfrac{\pi}{|y|} |z| \bigr)^{-n}, \\
	  \muc{sec.trp}{4} := 4^{\sigma+1} \max(1,\tfrac{\pi}{|y|})^{2\sigma} N \bigl(1-(\tfrac{N}{N+1})^N\bigr) \muc{sec.trp}{2},
	  \end{gathered}\\
	 \begin{gathered}\label{trp.U**est2as}
	  \text{for} \;\; p = p^\varepsilon(p_\bot,z,t) \in \mathcal B^d_r \cap Z^\varepsilon_y, \;\;\;
	  0 < \varepsilon < \min\{ 1, \tfrac 1 2 |y|, \tfrac{\pi}{N+1} \}, \\
	1 \leq r \leq 2\sqrt E - \tfrac{\pi}{|y|}, \; E \geq \max \bigl( \muc{sec.trp}{3} r^{4\sigma} \varepsilon^{-2} , E^* \bigr).
 \end{gathered}
\end{gather}
\end{proposition}

\begin{proof}[Proof of Proposition \ref{trp.propest}.] 
As in the proof of Proposition \ref{tcp.propest}, we have formula \eqref{tcp.U**est2}.

Besides, it follows from \eqref{pex.zetac} that
\begin{equation}\label{trp.zetaest}
  \begin{aligned}
  &|\zeta_{\widehat{w_1},\widehat{w_2}}(p)| \geq 4^{-\sigma} \tfrac{2}{\pi} \pmuc{sec.pex}{1}^2 r^{-2\sigma} \varepsilon, \\
  &p \in \mathcal B^d_r \setminus Z^\varepsilon_y, \quad 0 < \varepsilon < 1, \quad 1 \leq r \leq 2 \sqrt E.
  \end{aligned}
\end{equation}
Using \eqref{trp.zetaest} and \eqref{tcp.U**est2} with $\delta = 4^{-\sigma} \tfrac{2}{\pi} (\muc{sec.pex}{1})^2  r^{-2\sigma} \varepsilon$, we get \eqref{trp.U**est}.

It remains to prove \eqref{trp.U**est2}. Using definition \eqref{itr.defU**e}, one can write
\begin{equation}\label{trp.U**edec}
  \begin{aligned}
  &U^{**}_{N,\varepsilon}(p,E) - \widehat v(p) = \varphi^{N,\varepsilon}(p,E) + V^{N,\varepsilon}(p), \\
  &\mbox{for }p = p^\varepsilon(p_\bot,z,t) \in \mathcal B^d_r \cap Z^\varepsilon_y, \;\; r \leq 2 \sqrt E - \tfrac{\pi}{|y|} \mbox{with }\\
  &\varphi^{N,\varepsilon}(p,E) := \hspace{-1em} \sum_{-N \leq j \leq N, j \neq 0} \hspace{-1em} \left[ U^{**}\bigl(p^\varepsilon(p_\bot,z,j),E\bigr)-\widehat v\bigl(p^\varepsilon(p_\bot,z,j)\bigr) \right] L_j(t),
  \end{aligned}
\end{equation}
where $V^{N,\varepsilon}$ is given by \eqref{trp.defVe}.

Using estimate \eqref{trp.U**est} and formulas \eqref{itr.defU**e}, \eqref{itr.pjinset}, \eqref{trp.U**edec} we get
\begin{equation}\label{trp.phiest1}
 \begin{gathered}
	| \varphi^{N,\varepsilon}(p,E) | \leq \muc{sec.trp}{2} E^{-\alpha-\frac 1 2} \bigl(r+\tfrac{\pi}{|y|}\bigr)^{2\sigma} \varepsilon^{-1} \hspace{-1em} \sum_{-N\leq j \leq N,j \neq 0}|L_j(t)| \\
	\leq 4^\sigma\max(1,\tfrac{\pi}{|y|})^{2\sigma} \muc{sec.trp}{2} E^{-\alpha-\frac 1 2} \varepsilon^{-1} r^{2\sigma} \sum_{-N\leq j \leq N, j\neq 0} |L_j(t)|,\\
	  p = p^\varepsilon(p_\bot,z,t) \in \mathcal B^d_r \cap Z^\varepsilon_y, \quad 1 \leq r \leq 2\sqrt E - \tfrac{\pi}{|y|}.
 \end{gathered}
\end{equation}
Note that for $t \in (-1,1)$ and $j = \pm 1, \dots, \pm N$,
\begin{equation}
  \begin{aligned}
  |L_j(t)| & = \frac{|j||t+j|}{(N-j)!(N+j)!} \prod_{1\leq i \leq N, i\neq j}(i^2-t^2) \\ 
   & \leq \frac{2 N!N!}{(N-j)!(N+j)!} \leq 2 \left( \frac{N}{N+1} \right)^{|j|}.
  \end{aligned}
\end{equation}
Consequently,
\begin{equation}\label{itr.Ljsumest}
  \sum_{-N\leq j \leq N, j \neq 0} |L_j(t)| \leq 4N \bigl( 1 - \bigl(\tfrac{N}{N+1}\bigr)^N \bigr), \quad t \in (-1,1).
\end{equation}
Using \eqref{trp.phiest1}, \eqref{itr.Ljsumest}, we get
\begin{equation}\label{trp.phiest2}
  \begin{aligned}
	&|\varphi^{N,\varepsilon}(p,E)| \leq \muc{sec.trp}{4} E^{-\alpha-\frac 1 2} r^{2\sigma}  \varepsilon^{-1}, \\
&\mbox{for }	p = p^\varepsilon(p_\bot,z,t) \in \mathcal B^d_r \cap Z^\varepsilon_y, \quad 1 \leq r \leq 2\sqrt E - \tfrac{\pi}{|y|}, \\
	&E \geq \max\bigl( \muc{sec.trp}{3} r^{4\sigma} \varepsilon^{-2} , E^* \bigr).
  \end{aligned}
\end{equation}
Using \eqref{trp.Vest1}, \eqref{trp.U**edec}, \eqref{trp.phiest2}, we get \eqref{trp.U**est2}.

Proposition \ref{trp.propest} is proved.
\end{proof}

The final part of the proof of Theorem \ref{itr.thmr} is as follows. In a similar way with \eqref{tcp.vdec}, we represent $v$ as follows:
\begin{equation}\label{trp.vdec}
 \begin{gathered}
    v(x) = v^+_1(x,r) + v^+_2(x,r) + v^-(x,r), \quad x \in D, \; r > 0 \;\mbox{with}\\
    v^+_1(x,r) := \int_{\mathcal B^d_r \setminus Z^\varepsilon_y} e^{-ipx} \widehat v(p) \, dp, \\
    v^+_2(x,r) := \int_{\mathcal B^d_r \cap Z^\varepsilon_y} e^{-ipx} \widehat v(p) \, dp, \\
    v^-(x,r) := \int_{\mathbb R^d \setminus \mathcal B^d_r} e^{-ipx} \widehat v(p) \, dp.
 \end{gathered}
\end{equation}
Since $v \in W^{n,1}(\mathbb R^d)$, estimate \eqref{tcp.v-est} holds.

Using estimates \eqref{trp.U**est}, \eqref{trp.U**est2}, we get:
\begin{gather}\label{trp.v+est}
  \begin{gathered}
      \biggl| v^+_1(x,r)-\int_{\mathcal B^d_r \setminus Z^\varepsilon_y}   e^{-ipx} U^{**}(p,E) \, dp 
        + v^+_2(x,r) - \int_{\mathcal B^d_r \cap Z^\varepsilon_y }   e^{-ipx} U^{**}_{N,\varepsilon}(p,E) \, dp \biggr| \leq I_1 + I_2, \\
  \end{gathered}\\
   I_1 := 2^{2\sigma} \muc{sec.trp}{2}|\mathcal B^d_1| E^{-\alpha-\frac 1 2} r^{d+2\sigma} \varepsilon^{-1}, \label{trp.I1est}\\
   I_2 := \muc{sec.trp}{1} \varepsilon^{2N} \int_{\mathcal B^d_r \cap Z^\varepsilon_y} \hspace{-2em} \bigl(1+ \tfrac{\pi}{|y|} |z(p)| + |p_\bot| \bigr)^{-n} \, dp,\\
      x \in D, \; 1 \leq r \leq 2 \sqrt E - \tfrac{\pi}{|y|}, \; \text{for $E$ as in \eqref{trp.U**est2as}}. \notag
\end{gather}
In addition, we have the following estimate proved in \cite{Agal2016b} (see the proof of Theorem 2 of \cite{Agal2016b}):
\begin{equation}\label{trp.I2est}
    \begin{gathered}
   I_2 \leq \muc{sec.trp}{5} \varepsilon^{2N+1}, \quad \muc{sec.trp}{5} :=  \muc{sec.trp}{1}  \frac{2}{|y|}\frac{|\mathbb S^{d-2}|}{n-d+1} \sum_{z \in \mathbb Z} (1 + \tfrac{\pi}{|y|} |z| )^{d-n-1}.
    \end{gathered}
\end{equation}

In addition, if $r = r(E)$, $\varepsilon = \varepsilon(E)$, where $r(E)$, $\varepsilon(E)$ are defined in \eqref{itr.beta2def}, then
\begin{equation}\label{trp.rviaE}
  \begin{aligned}
      r^{d-n} & = (2\tau)^{d-n} E^{-\beta}, \\
      \varepsilon^{-1}E^{-\alpha-\frac 1 2} r^{d+2\sigma} & = (2\tau)^{d+2\sigma} E^{-\beta}, \\
      \varepsilon^{2N+1} & = E^{-\beta}.
  \end{aligned}
\end{equation}

Next, we establish a lower bound for $E$ for which $r = r(E)$, $\varepsilon = \varepsilon(E)$ satisfy conditions \eqref{trp.U**estas}, \eqref{trp.U**est2as} of Proposition \ref{trp.propest}. Note that
\begin{equation}
  \begin{gathered}\label{trp.Evalid}
	E \geq \muc{sec.trp}{3} r(E)^{4\sigma} \varepsilon(E)^{-2} \quad \text{if $E \geq \muc{sec.trp}{6}$},\\
	\muc{sec.trp}{6} := \left(\muc{sec.trp}{3}(2\tau)^{4\sigma}\right)^{\frac{(n+2\sigma)(2N+1)+n-d}{(n-4\alpha\sigma)(2N+1)-2\alpha(n-d)}}. 
  \end{gathered}
\end{equation}
Besides,
\begin{equation}\label{trp.rEvalid}
	r(E) \leq 2 \sqrt E - \tfrac{\pi}{|y|} \quad \text{if $E \geq 1$}.
\end{equation}
In addition,
\begin{equation}\label{trp.eEvalid}
  \begin{gathered}
   \varepsilon(E) < \min\{1,\tfrac 1 2 |y|, \tfrac{\pi}{N+1}\} \quad \text{if $E > \muc{sec.trp}{7}$}, \\
	\muc{sec.trp}{7} := \bigl( \max\{1,\tfrac{2}{|y|},\tfrac{N+1}{\pi} \}  \bigr)^\frac{2N+1}{\beta}.
   \end{gathered}
\end{equation}

Using the representation \eqref{trp.vdec} and formulas \eqref{itr.beta2def}, \eqref{tcp.v-est}, \eqref{trp.v+est}, \eqref{trp.I1est}, \eqref{trp.I2est}--\eqref{trp.eEvalid} we get
 \begin{equation}\label{trp.B2E2def}
  \begin{aligned}
	&\bigl| v^{**}_E(x) - v(x) \bigr| \leq B_2 E^{-\beta}\quad \mbox{with} \\
	&B_2 := \muc{sec.tcp}{3} (2\tau)^{d-n} \|v\|_{n,1} + 2^{2\sigma} \muc{sec.trp}{2} |\mathcal B^d_1|(2\tau)^{d+2\sigma} + \muc{sec.trp}{5},\\
	&E_2 := \max\{ \muc{sec.trp}{6}, \muc{sec.trp}{7}, E^*\}.
	\end{aligned}
 \end{equation}

Theorem \ref{itr.thmr} is proved.

\paragraph*{Acknowledgement:} TH gratefully acknowledges financial support by DFG through grant 
CRC 755/C02.

\bibliographystyle{plain}
\bibliography{biblio_utf}    

\end{document}